\documentclass[10pt]{article}
\usepackage{amssymb}
\usepackage[latin1]{inputenc}
\usepackage{graphicx,color}
\usepackage{tikz}
\usepackage{calrsfs}
\def\nd{\noindent}

\usepackage{color}
\newtheorem{theorem}{Theorem}[section]
\newtheorem{thmlet}{Theorem}

\newtheorem{definition}{Definition}[section]
\newtheorem{lemma}{Lemma}[section]
\newtheorem{proposition}{Proposition}[section]

\newtheorem{corollary}{Corollary}[section]
\newcommand{\fim}{\hfill\rule{2mm}{2mm}}

\begin{document}
\title{
\vspace{0.5in} \textbf{\textit{Continuums}} {\bf  of positive solutions  for  classes of non-autonomous and non-local  problems with strong singular term}}

\author{
{\bf\large Carlos Alberto Santos$^{a,}$}\footnote{Carlos Alberto Santos acknowledges
the support of CAPES/Brazil Proc.  $N^o$ $2788/2015-02$.}\,\, ~~~~~~ {\bf\large Lais Santos$^b$}  ~~~~~~ {\bf\large Pawan Kumar Mishra$^{c,}$}\footnote{Pawan Mishra acknowledges
the support of CAPES/Brazil.}
\hspace{2mm}\\
{\it\small $^a$Universidade de Bras\'ilia, Departamento de Matem\'atica}\\
{\it\small   70910-900, Bras\'ilia - DF - Brazil}\\
{\it\small $^b$Departamento de Matem\'atica, Universidade de Vi\c cosa}\\
{\it\small   CEP 36570.000, Vi\c cosa - MG, Brazil}\\
{\it\small $^c$Universidade Federal da Para\'iba, Departamento de Matem\'atica, Cidade Unversit\'aria- Campus I}\\
{\it\small   58051-900, Jo\~ao Pessoa- PB - Brazil}\\
{\it\small e-mails: csantos@unb.br,
matmslais@gmail.com, pawanmishra31284@gmail.com }\vspace{1mm}\\
}

\date{}
\maketitle \vspace{-0.2cm}

\begin{abstract}
\mbox{}

In this paper, we show existence of \textit{continuums} of positive solutions  for  non-local quasilinear problems with strongly-singular reaction term  on a bounded domain in $\mathbb{R}^N$ with $N \geq 2$. We approached  non-autonomous and non-local equations by applying the Bifurcation Theory to the corresponding $\epsilon$-perturbed problems and using a comparison principle for $W_{\mathrm{loc}}^{1,p}(\Omega)$-sub and supersolutions to obtain qualitative properties of the $\epsilon$-\textit{continuum} limit. Moreover, this technique empowers us to study a strongly-singular and non-homogeneous Kirchhoff problem to get the existence of a \textit{continuum} of positive solutions.

\end{abstract}

\nd {\it \footnotesize 2010 Mathematics Subject Classifications:} {\scriptsize 35J25, 35J62, 35J75, 35J92}\\
\nd {\it \footnotesize Key words}: {\scriptsize Strongly-singular nonlinearities, Non-local Kirchhoff problems, Comparison Principle for $W_{\mathrm{loc}}^{1,p}(\Omega)$-sub and supersolutions, Bifurcation theory. }

\section{Introduction}
\def\theequation{1.\arabic{equation}}\makeatother
\setcounter{equation}{0}

In this paper, we study the existence, multiplicity and non-existence of positive $W_{\mathrm{loc}}^{1,p}(\Omega)$-solutions for the following non-autonomous and non-local $\lambda$-problem
$$
 (P)\left\{
\begin{array}{l}
-A\Big(x, \displaystyle\int_{\Omega}u^{\gamma}dx\Big)\Delta_pu = {\lambda }f(x,u)  ~ \mbox{in } \Omega,\\
    u>0 ~ \mbox{in }\Omega,~~
    u=0  ~ \mbox{on }\partial\Omega,
\end{array}
\right.
$$
where $ \Omega \subset \mathbb{R}^N ( N \geq 2)$ is a smooth bounded domain, $p \in (1, N), \ -\Delta_pu = -\mbox{div}(|\nabla u|^{p-2}\nabla u)$ is the $p-$Laplacian operator,   $\lambda >0$ is a real parameter, $A \in C(\overline{\Omega}\times [0,\infty), (0, \infty))$ and $f \in C(\overline{\Omega}\times(0, \infty), (0, \infty))$ can be strongly (very) singular at $u=0$.

The problem $(P)$ is non-local due to the presence of the term $A \Big(x, \int_{\Omega}u^{\gamma} \Big)$, which implies that the equation in $(P)$ is no longer a pointwise equality. Problems like $ (P) $ arise in various physical and biological models, for example, in the studies of particles in thermodynamical equilibrium via gravitational potential \cite{MR1380234}, 2-D turbulent behavior of real flow \cite{MR1145596}, physics of plasmas and population dynamics. In \cite{MR1603446}, it was investigated that the equation
\begin{equation}\label{1}
\frac{du}{dt} - A\Big(\int_{\Omega} u\Big)\Delta u = f
\end{equation}
 describes, for instance, the behavior of a population subject to some kind of spreading. In this case, $u$ and $A$ represent the population density and the diffusion coefficient, respectively. When $A$ is a constant, the above model does not take in to account that the phenomena of crowding and isolation can change the dynamics of the migration. Therefore, in a closer model  to  the reality, the coefficient $A$ is supposed to depend on the entire population in the domain $\Omega$ as in (\ref{1}).

 The literature about non-local problems with an autonomous and non-local term is vast (see, for example, \cite{MR3101144}, \cite{MR2409464}, \cite{MR2098510}, \cite{MR2012587} and \cite{MR1895897} ). Due to the lack of variational structure, non-local problems such as $(P)$ are treated, in general,  through  topological methods.  A recurrent argument in the treatment of autonomous non-local problems is to relate the non-local problem to a local problem and from that to study the behavior of the associated local problem. For example, in \cite{MR3101144} the authors have considered the following class of problems
 \begin{equation}\label{2}
 -\Big(\int_{\Omega}g(u)dx\Big)^r \Delta u = \lambda f(u) ~\mbox{in} ~ \Omega, ~u > 0 ~\mbox{in} ~\Omega  ~ \mbox{and} ~u = 0 ~\mbox{on} ~\partial \Omega,
 \end{equation}
 in which $u$  solves (\ref{2}) if and only if it is a positive solution to
 \begin{eqnarray}\label{3}
 - \Delta u = \alpha f(u) ~\mbox{in},  ~u = 0 ~\mbox{on} ~\partial \Omega,
\end{eqnarray}
with $	\alpha = \lambda \Big(\int_{\Omega}g(u)dx\Big)^{-r}. $
Therefore, results of existence and multiplicity to (\ref{2}) were proved through the study of $h_{\lambda}(\alpha, u) = \alpha - \lambda \Big(\int_{\Omega}g(u)dx\Big)^{-r},$  constrained to a \textit{continuum} of solutions of (\ref{3}).  This type of arguments, in general, can not be applied for non-autonomous and non-local problems. There are a few papers on the non-autonomous case, see \cite{MR2755414}, \cite{MR3745736}, \cite{MR3694626} and references therein. In particular, we mention \cite{MR3694626} where the problem $(P)$ is treated via bifurcation theory with $p = 2$ and $f(x,u) = u^{\beta}$,  for $ 0 < \beta < 1$.

In this work, we are concerned principally to $(P)$ with $f(x,t)$ being strongly singular at $t = 0$.{ Non-local singular problems have already been treated in the literature  when $f$ is weakly singular at $t=0$ (i.e. $f(x,t) \approx t^{\delta}$ with $-1<\delta <0$) or in the context of classical solutions, see \cite{MR3316619}, \cite{MR3665600} and references therein}. To our knowledge, \cite{NOSSO} was the first work to consider a non-local and strongly-singular quasilinear problems. However, with a monotonicity condition on $f (x, t)/ t^{p-1}$, a uniqueness result was shown there and as a consequence of this, the analysis of the behavior of the {\textit{continuum}}  was done by studying the {parameter-solution application}.

In this paper, since $A$ is a non-autonomous function and no monotonicity is posed on the quotient $f (x, t)/ t^{p-1}$, the same strategy can not be applied anymore.  In \cite{MR0427826}, Rabinowitz et. al. studied semilinear local singular problems in the context of classical solutions.  We inspire our approach on ideas from  them to obtain an unbounded $\epsilon$-limit connected component of positive solutions from $\epsilon$-unbounded {\textit{continuum}} of positive solutions for a $\epsilon$-perturbed problems. For qualitative properties about this \textit{continuum}, we are inspired on ideas from Figueiredo-Sousa et. al. \cite{MR3694626}, where a semilinear non-local problem was treated with non-singular (sublinear) growth. The same strategies of both above papers do not work in our approach,  principally by the lack of  the linearity of the $p$-Laplacian operator and {by the singularity in the Sobolev spaces setting}. To overcome these difficulties, we approached $(P)$ in an indirect way, since no functional equation can be directly associated to $(P)$, by  combining penalization arguments, a-priori estimates and a recent Comparison Principle for $W_{\mathrm{loc}}^{1,p}(\Omega)$-sub and supersolutions, proved by the two first authors of this paper (see Theorem 2.1 in \cite{NOSSO}).

Before stating the main results of this work, we need to clarify what we mean by Dirichlet boundary condition and solution to $(P)$. After the remarkable paper of Mackenna \cite{MR1037213}, we know that a solution of the problem $(P)$, with $p=2$, $A \equiv 1$ and $f(x,t) = t^{-\delta}$, lies in $H_0^{1}(\Omega)$ if and only if $0 < \delta < 3$. Therefore, for stronger singularities, we need a more general concept of zero-boundary conditions.

\begin{definition}\label{D1} We say that $u \leq 0$ on $\partial \Omega$ if $(u-\epsilon)^+ \in W_0^{1,p}(\Omega)$ for every $\epsilon > 0$ given. Furthermore, $u \geq 0$ if $-u \leq 0$ and $u = 0$ on $\partial \Omega$ if $u$ is a non-negative and non-positive function in $\partial \Omega$.
\end{definition}

In the following, we define a solution of the problem $(P)$.

\begin{definition}\label{D2} We say that $u$ is a $ W_{\mathrm{loc}}^{1,p}(\Omega)$-solution for $(P)$ if
 $u > 0$ in $\Omega$, that is, for each $\Theta \subset \subset \Omega$ given there exists a positive constant $c_\Theta$ such that $u \geq c_\Theta > 0$ in $\Theta$, $u^{\gamma} \in L^1(\Omega)$ and
\begin{equation}\label{4}
\displaystyle\int_{\Omega}|\nabla u|^{p-2}\nabla u \nabla \varphi dx = \lambda \displaystyle\int_{\Omega} \frac{f(x, u)}{A\Big(x, \int_{\Omega}u^{\gamma}dx\Big)}\varphi dx ~~\mbox{for all} ~ \varphi \in C_c^{\infty}(\Omega).
\end{equation}
\end{definition}

In what follows, we will always assume that $f \in C(\overline{\Omega}\times (0, \infty), (0, \infty))$. Let us set some hypotheses that we need in our first Theorem.
\begin{itemize}
\item[$(A_0)$] $A \in C(\overline{\Omega}\times \mathbb{R})$ satisfy $A(x,t) > 0 $ for all $ t \geq 0$ and $x \in \overline{\Omega}, $
\noindent \item[$(f_0)$] $\displaystyle\lim_{t \rightarrow 0^+}\frac{f(x,t)}{t^{p-1}} = \infty $ uniformly in $\overline{\Omega}$,
\noindent\item[$(f_\infty)$] $\displaystyle\lim_{t \rightarrow \infty }\frac{f(x,t)}{t^{p-1}} = 0$ uniformly in $\overline{\Omega}$.
\end{itemize}
Our first result can be stated as follows.
\begin{theorem}\label{T1} Suppose that $\gamma \geq 0$, $(A_0)$ and $(f_0)$ hold. Then, there exists an unbounded \textit{continuum} $\Sigma \subset \mathbb{R} \times C(\overline{\Omega})$ of positive solutions of $(P)$  that emanates from $(0,0)$.  In additional, if $(f_\infty)$ holds and $ A(x,t) \geq a_0  ~\mbox{in} ~\overline{\Omega}\times \mathbb{R}^+$ for some $a_0>0$, then $Proj_{\mathbb{R}}\Sigma = (0, \infty)$.
\end{theorem}

Below, we present more   qualitative information about the \textit{continuum} $\Sigma $ by relating the non-local and {nonlinear} terms. In this case, we need to consider certain additional conditions:

\vspace{0.2cm}
\noindent $(A_\infty)$~ $\displaystyle\lim_{t \rightarrow \infty} A(x,t)t^{\theta} = a_\infty(x) \geq 0 ~\mbox{uniformly in } ~ \overline{\Omega}, ~ \mbox{for some} ~ a_\infty \in C(\overline{\Omega})$, 
\vspace{0.2cm}

\noindent $(A'_\infty)$~ $\displaystyle\lim_{t \rightarrow \infty} A(x,t)t^{\theta} = \infty  ~\mbox{uniformly in } ~ \overline{\Omega},$ 
\vspace{0.2cm}

\noindent $(f_1)$ $\displaystyle\lim_{t \rightarrow \infty}\frac{f(x,t)}{t^{\beta}} = c_\infty(x) > 0 $ uniformly in $\overline{\Omega}$, for some $ -\infty < \beta < p-1$ and $c_\infty \in C(\overline{\Omega})$, 
\vspace{0.2cm}

\noindent $(f_2)$ $\displaystyle\lim_{t \rightarrow 0^+}\frac{f(x,t)}{t^{\delta}} = c_0(x) > 0 $ uniformly in $\overline{\Omega}$, for some $ -\infty < \delta < p-1$ and $c_0 \in C(\overline{\Omega}). $

 \begin{theorem}\label{T2}
Assume $(A_0)$ and that $f$ satisfies $(f_1)$ and $(f_2)$, with $ \delta \leq \beta $. If
\begin{itemize}
\item[a)] $\gamma > 0$ and either $\{$$\theta{\gamma} = p-1-\beta$ and $ (A'_\infty)$$\}$ or $\{$$\theta{\gamma} <p-1-\beta$ and $ (A_\infty)$ with $a_\infty>0$ in $\overline{\Omega}$$\}$ hold, then $Proj_{\mathbb{R}}\Sigma = (0, \infty)$ (see Fig. 1),
\item[b)] $\gamma > 0$, $\theta{\gamma} \geq {p-1-\beta}$ and $(A_\infty)$ hold, then $Proj_{\mathbb{R}}\Sigma \subset (0, \lambda^*)$ for some $0 < \lambda^* < \infty. $ Besides this, if
\begin{itemize}
\item[i)]   $\theta{\gamma} = {p-1-\beta}$ and  $a_\infty > 0 $ in $\overline{\Omega}$, then $\lambda = 0$ can not be a bifurcation point from $\infty$ (see Fig. 2 or 3);
\item[ii)]  $a_\infty = 0$ in $\overline{\Omega}$, then $\lambda = 0$ is a bifurcation point from $\infty$ (see Fig. 4);
\end{itemize}
\item[c)] $-1 < \gamma < 0$, $\theta{\gamma} \geq {p-1-  \delta} $ and either $(A'_\infty)$ or $(A_\infty)$  with $0 < a_\infty $ hold, then $(P)$  does not admit positive solution for $\lambda > 0$ small.
\end{itemize}
 \end{theorem}
 Summarizing the above information, we have the following diagrams.
 \newpage
\begin{figure}
	\centering
	\begin{tikzpicture}[scale=.55]
	\draw[thick, ->] (-1, 0) -- (8, 0);
	\draw[thick, ->] (0, -1) -- (0, 7);
    \draw[thick] (0, 0) .. controls (1, 3) and (4, 0) ..(7,6.5);
	\draw (8,0) node[below]{$\lambda$};
	\draw (0, 0) node[below left]{$0$};
	\draw (0,7) node[left]{$\|u\|_\infty$};
   \draw (3,-1) node[below]{\textrm{Fig. 1 Theorem 1.2 a)}};
    \end{tikzpicture}
    \hspace{0.8cm}
    \begin{tikzpicture}[scale=0.55]
	\draw[thick, ->] (-1, 0) -- (8, 0);
	\draw[thick, ->] (0, -1) -- (0, 7);
    \draw[thick] (0, 0) .. controls (6, 2) and (0.5, 3) ..(1,7);
	\draw[thick, dotted] (0.8,0) -- (0.8,7);
	\draw (8,0) node[below]{$\lambda$};
	\draw (0, 0) node[below left]{$0$};
	\draw (0,7) node[left]{$\|u\|_\infty$};
\draw (3,-1) node[below]{\textrm{Fig. 2 Theorem 1.2 b-i)}};
	\end{tikzpicture}
    \hspace{0.8cm}
    \begin{tikzpicture}[scale=0.55]
\draw[thick, ->] (-1, 0) -- (8, 0);
	\draw[thick, ->] (0, -1) -- (0, 7);
	\draw[thick] (0 ,0) .. controls (0.8, 1.9) and (2.7,0.5) .. (2.9, 7) ;
	\draw[thick, dotted] (3,0) -- (3, 7.0);
	\draw (8,0) node[below]{$\lambda$};
	\draw (0, 0) node[below left]{$0$};
	\draw (0,7) node[left]{$\|u\|_\infty$};
\draw (3,-1) node[below]{\textrm{Fig. 3 Theorem 1.2 b-i)}};
	\end{tikzpicture}	
	\hspace{0.8cm}
	\begin{tikzpicture}[scale=0.55]
   \draw[thick, ->] (-1, 0) -- (8, 0);
	\draw[thick, ->] (0, -1) -- (0, 7);
	\draw[thick] (0.0 ,0.0) .. controls (5.0,3.0) and (0.2,3) .. (0.2, 6.8);
	\draw (8,0) node[below]{$\lambda$};
	\draw (0, 0) node[below left]{$0$};
	\draw (0,7) node[left]{$\|u\|_\infty$};
   \draw (3,-1) node[below]{\textrm{Fig. 4 Theorem 1.2 b-ii)}};
    \end{tikzpicture}
\end{figure}

 In the above item $(c)$, we stated that the problem $(P)$ has no solution {for} $ \lambda>0 $ close to $0$ when the non-local term is also singular. We note that the issue about existence of solution is not possible to treat no longer with the same arguments as used in the proof of Theorem \ref{T1}. However, when the non-local term is autonomous, we are also able to prove the global existence of $ W_{\mathrm{loc}}^{1,p}(\Omega) \cap C(\overline{\Omega})$-solutions.

 More precisely, we have the following result.

\begin{theorem}\label{T4}  Assume that $(f_1)$, $(f_2)$ with $\delta \leq \beta$, $(A_0)$ and either $(A_\infty)$  with $ a_\infty >0$ or $(A'_\infty)$  hold. If $\gamma\theta > p-1-\delta$ and $-1 < \gamma < 0$, then there exists a $ \lambda^*>0 $ such that  the problem
\begin{equation}\label{autonomo}
 \left\{
\begin{array}{l}
-A\Big(\displaystyle\int_{\Omega}u^{\gamma}dx\Big)\Delta_pu = {\lambda }f(x,u)  ~ \mbox{in } \Omega,\\
    u>0 ~ \mbox{in }\Omega,~~
    u=0  ~ \mbox{on }\partial\Omega,
\end{array}
\right.
\end{equation}
 admits at least one $W^{1,p}_{\mathrm{loc}}(\Omega)\cap C(\overline{\Omega})$-solution for $\lambda \geq \lambda^*$ and no solution for $\lambda <\lambda^*$.
\end{theorem}

By taking advantage on the ideas explored in the proofs of the above Theorems, we  were able to consider non-autonomous Kirchhoff-type problems as well.  For sake of the clarity, let us consider just a classical Kirchhoff model.
Precisely, we consider
$$
 (Q)~~\left\{
\begin{array}{l}
-M\Big(x, \|\nabla u\|_p^p)\Delta_pu = {\lambda }f(x,u)  ~ \mbox{in } \Omega,\\
    u>0 ~ \mbox{in }\Omega,~~
    u=0  ~ \mbox{on }\partial\Omega,
\end{array}
\right.
$$
where $M$, modeled as non-homogeneous Kirchhoff term, satisfies:
\begin{itemize}
\item[($M_0$)] $M(x,t) = a(x) + b(x)t^{\gamma}$, $a, b \in C(\overline{\Omega}),  a(x)   \geq \underline{a}$ and  $b(x) \geq 0$ in $\overline{\Omega}$
\end{itemize}
and $\gamma>0$ satisfies:
\begin{itemize}
\item[($\Gamma_0$)] $
\gamma > 0 ~\mbox{if} ~ -1 \leq \delta < p-1 ~\mbox{and} ~
0 < \gamma < \frac{p-1-\delta}{-\delta -1} ~\mbox{if} ~ -\frac{2p-1}{p-1} \leq \delta < -1.
$
\end{itemize}

\begin{theorem}\label{T5}
Assume that  $(f_2)$, $(M_0)$ and $(\Gamma_0)$ hold. Then there exists an unbounded \textit{continuum} $\Sigma \subset  \mathbb{R}^+ \times C(\overline{\Omega})$ of solutions of $(Q)$ which emanates from $(0,0)$. Besides this, if $(f_\infty)$
holds then $Proj_{\mathbb{R}^+}\Sigma = (0, \infty)$. Moreover, if $\gamma < 1$ then $\Sigma$ is unbounded vertically as well.
\end{theorem}

We remark that there are few papers dealing with  Kirchhoff type problems with singular nonlinearity. In this direction, we found some results in  \cite{MR3250494} and \cite{MR3352001}  for weak singularities, that permitted them to approach by variational methods. Recently, in 2018, Agarwal, O'Regal and Yan \cite{arg} studied a Kirchhoff-type problem with nonlinearity of the form $f (x, u) = K(x)u^{\delta}$, for $\delta < 0$, in the context of the Laplacian operator. They used principally  sub-supersolution techniques to get existence and uniqueness of classical solution.

It is worth mentioning that, as far as we know, non-autonomous and  non-local quasilinear problems with very singular nonlinearities have not yet been considered in the literature, and the same is true for Kirchhoff-type problems. Our results contribute to the literature principally by:
\begin{itemize}
\item[$i)$] Theorem \ref{T1}, being new even in the context of local problems (and for $p=2$), by guaranteeing the existence of a \textit{continuum} of solutions for a strongly-singular problem in the weak solutions setting.  Moreover, the conclusion that this \textit{continuum} is horizontally unbounded is obtained without any boundedness condition on $f$, contrary to Theorem 1.9 and Corollary 1.10 in \cite{MR0427826},
\item[$ii)$] Theorem \ref{T2}, proving the principal results of Figueiredo-Sousa et. al. \cite{MR3694626} in the context of strongly-singular problems as well,
\item[$iii)$] Theorem \ref{T4}, including singularity also in the non-local term and obtaining global existence of solutions in $W^{1,p}_{\mathrm{loc}}(\Omega)\cap C(\overline{\Omega})$ setting. This situation was not yet considered in the literature,
\item[$iv)$] Theorem \ref{T5}, including non-autonomous Kirchhoff terms and capturing the same sharp power  for existence of solutions still in $W_0^{1,p}(\Omega)$ for the associated local problem.
\end{itemize}

Our work follows the following structure. In the second section, we present the proof of Theorem \ref{T1}. In section 3, we establish the fundamental tools for our approach. The qualitative  study of the \textit{continuum} obtained in the second section will be done in section 4, as well  the proof of Theorem \ref{T4}. We conclude the section 4, by studying  the degenerate case in problem $(P)$. In the last section we prove Theorem \ref{T5}.
\vspace{0.2cm}

Throughout this paper, we make use of the following notations:
\begin{itemize}
\item The norms in $L^p(\Omega)$ and $W_0^{1,p}(\Omega)$ are denoted by $\|u\|_p$ and $\|\nabla u\|_p,$ respectively.
\item $C_c^{\infty}(\Omega) = \{ u : \Omega \rightarrow \mathbb{R}: u \in C^{\infty}(\Omega) \ \mbox{and} \ supp ~ u \subset \subset \Omega \}$.
\item $B_R(\lambda_0, u_0) = \{(\lambda, u) \in \mathbb{R} \times C(\overline{\Omega}) : |\lambda - \lambda_0| + \Vert u - u_0\Vert_{\infty} < R\}. $
\end{itemize}

\section{Proof of Theorem \ref{T1}}
Throughout this paper, we will denote by $ e \in C_0^1(\overline{\Omega})$ the unique positive solution of
$$
-\Delta_p u = 1 ~~ \mbox{in} ~\Omega, ~u|_{\partial \Omega} = 0
$$
and by $\phi_1 \in C_0^1(\overline{\Omega})$ the first positive normalized eigenfunction associated to the first positive eigenvalue of $(-\Delta_p, W_0^{1,p}(\Omega))$, that is,
$$-\Delta_p \phi_1 = \lambda_1\phi_1^{p-1} ~~ \mbox{in} ~\Omega, ~\phi_1|_{\partial \Omega} = 0.
$$

For each $\epsilon>0$ given, let us introduce the following $\epsilon$-perturbed problem
$$
 (P_{\epsilon})~ ~ ~ \left\{
\begin{array}{l}
-A\Big(x, \displaystyle\int_{\Omega}u^{\gamma}dx\Big)\Delta_pu = {\lambda }f(x,u + \epsilon)  ~ \mbox{in } \Omega,\\
    u>0 ~ \mbox{in }\Omega,~~
    u=0  ~ \mbox{on }\partial\Omega
\end{array}
\right.
$$
and show that  $(P_{\epsilon})$ admits an unbounded $\epsilon$-\textit{continuum} of positive solutions by using the Rabinowitz Global Bifurcation Theorem, more specifically Theorem 3.2 in \cite{MR0301587}.

 \begin{lemma}\label{L1} Suppose that  $\gamma \geq 0$ and $(A_0)$ hold.  Then there exists an unbounded \textit{continuum} $\Sigma_{\epsilon} \subset \mathbb{R}^+ \times C(\overline{\Omega})$ of positive solutions of  $(P_{\epsilon})$ that emanates from $(0,0), $ for each $\epsilon > 0$ given.
 \end{lemma}
 \begin{proof}
 It follows from the classical theory of existence and regularity for elliptic equations and hypothesis $(A_0)$ that the problem
 \begin{equation}\label{6}
 -A\Big(x, \displaystyle\int_{\Omega}|v|^{\gamma}dx\Big)\Delta_pu = {\lambda }f(x,|v| + \epsilon)  ~ \mbox{in } \Omega,  ~ ~   u=0  ~ \mbox{on } ~ \partial\Omega
 \end{equation}
admits a unique solution  $u \in C^{1, \alpha}(\overline{\Omega}),$ for some $\alpha \in (0,1)$ and for each $ (\lambda, v)  \in \mathbb{R}^+ \times C(\overline{\Omega})$. Thus, the operator $T: \mathbb{R}^+ \times C(\overline{\Omega}) \rightarrow C(\overline{\Omega}),$ which associates  each pair $(\lambda, v) \in \mathbb{R}^+ \times C(\overline{\Omega})$ to the only weak solution of (\ref{6}), is well-defined.

It is classical to show  that $T$ is a compact operator, using Arzel\`a-Ascoli's Theorem. Hence, we are able to apply Theorem 3.2 of \cite{MR0301587} to get  an unbounded $\epsilon$-\textit{continuum} $\Sigma_{\epsilon} \subset \mathbb{R}^+ \times C(\overline{\Omega})$ of solutions of
\begin{equation}\label{7}
T(\lambda, u) = u.
\end{equation}
 Moreover, by the definition $T(0, v) = 0$ and if $T(\lambda, 0) = 0$ implies $\lambda = 0$, we can conclude that $\Sigma_{\epsilon} \backslash \{(0,0)\}$  is formed by nontrivial solutions of (\ref{7}).

 Finally, using that $0<f(x, |v| + \epsilon)/A\Big(x, \int_{\Omega}|v|^{\gamma}\Big) \in L^{\infty} (\Omega)$ for each given $v \in C(\overline{\Omega})$ and classical strong maximum principle, we obtain that $T((\mathbb{R}^+\backslash \{0\}) \times C(\overline{\Omega})) \subset C(\overline{\Omega})_+$, where $C(\overline{\Omega})_+ = \{ u \in C(\overline{\Omega}) : u > 0 ~\mbox{in} ~ \Omega\}$. Therefore, $\Sigma_{\epsilon}$ is a $\epsilon$-\textit{continuum} of positive solutions of $(P_{\epsilon}), $  for each $\epsilon>0$ given. This ends the proof.
\end{proof}
\fim
\vspace{0.4cm}

As a consequence of  the result we just proved,   for every $\epsilon > 0$ and  for each bounded open set $U \subset \mathbb{R}\times C(\overline{\Omega})$  containing $(0,0)$, there exists a pair $ (\lambda_\epsilon, u_\epsilon) \in \Sigma_{\epsilon} \cap \partial U$. An essential argument in our approach is to show that
 if $ \epsilon_n \rightarrow 0^+ $ and $\lambda_n \rightarrow \lambda,$ then $ \lambda> 0 $ and $ \{u_{\epsilon_n} \} $ converges in $ C(\overline{\Omega})$ to a function $ u \in W_{\mathrm{loc}}^{1,p}(\Omega) \cap C(\overline{\Omega})$, where $(\lambda, u)$ is a solution of $(P).$

To prove this,  let us begin with the following result which is motivated by the arguments of Crandall, Rabinowitz and Tartar \cite{MR0427826}.
\begin{lemma}\label{L2} Admit that $(A_0)$ and $(f_0)$ hold. Let $U \subset \mathbb{R} \times C(\overline{\Omega})$ be a bounded open set containing $(0,0)$, a positive constant $K$ and a pair $(\lambda_{\epsilon}, u_{\epsilon}) \in \Big((0, \infty) \times (C(\overline{\Omega}) \cap W_0^{1,p}(\Omega))\Big) \cap \partial U$ of solution of $(P_{\epsilon})$ satisfying $\lambda_{\epsilon} \leq K$ and $u_{\epsilon} \leq K$ in $\overline{\Omega}$. Then, there exist constants $\mathcal{K}_1 = \mathcal{K}_1(K, U)>0$, $\mathcal{K}_2 = \mathcal{K}_2(k, K)>0$ and $\epsilon_0 > 0$ such that
\begin{equation}\label{8}
\lambda_{\epsilon}^{\frac{1}{p-1}}\mathcal{K}_1(K,U) \phi_1 \leq u_{\epsilon} \leq k + \lambda_{\epsilon}^{\frac{1}{p-1}}\mathcal{K}_2(k, K)^{\frac{1}{p-1}}e ~~ \mbox{in} ~\Omega,
\end{equation}
for each $k \in (0, K]$ fixed and for all $0 < \epsilon < \epsilon_0$.
\end{lemma}
\begin{proof} Let $K>0$ as above.  Besides this,
define $ 0 < a_K = \displaystyle\min_{\overline{\Omega}\times [0, ~ |\Omega|K^{\gamma}]} A(x, t)$ and $$\mathcal{K}_2(k, K) = \max\Big\{\frac{f(x,t)}{a_K}: ~ x \in \overline{\Omega} ~~\mbox{and} ~ k \leq t \leq K + 1\Big\}, $$ where $k$ is a fixed  number on  $(0, K ]$. Thus, $\mathcal{K}_2(k,\cdot)$ is non-decreasing for each $k$ fixed.

To show the second inequality in (\ref{8}), let us  consider the open set $\mathcal{O}_k = \{ x \in \Omega : u_{\epsilon} > k\}$. Then, it follows from the definition of $\mathcal{K}_2$ that
\begin{eqnarray*}
-\Delta_p \Big(  k + \lambda_{\epsilon}^{\frac{1}{p-1}}\mathcal{K}_2(k, K)^{\frac{1}{p-1}}e \Big) &=& \lambda_{\epsilon}\mathcal{K}_2(k, K) \geq \frac{\lambda_{\epsilon}}{a_K} f(x, u_{\epsilon}+ \epsilon)\\
&\geq & \frac{\lambda_{\epsilon}}{A\Big(x, \int_{\Omega} u_{\epsilon}^{\gamma}\Big)} f(x, u_{\epsilon}+\epsilon) =-\Delta_p u_\epsilon~~ \mbox{in} ~ \mathcal{O}_k .
\end{eqnarray*}
Since $  k + \lambda_{\epsilon}^{\frac{1}{p-1}}\mathcal{K}_2(k, K)^{\frac{1}{p-1}}e - u_{\epsilon} = \lambda_{\epsilon}^{\frac{1}{p-1}}\mathcal{K}_2(k, K)^{\frac{1}{p-1}}e \geq 0 $ on $\partial\mathcal{O}_k$ holds true, the claim is valid in $\mathcal{O}_k$ by classical comparison principle. Now, using the above fact together with the definition of $\mathcal{O}_k$, we conclude that $u_{\epsilon} \leq k + \lambda_{\epsilon}^{\frac{1}{p-1}}\mathcal{K}_2(k, K)^{\frac{1}{p-1}}e $ in $\Omega. $

Now, we are going to prove the first inequality in (\ref{8}). Let us denote by $ \delta' = \mbox{dist}(\partial U, (0,0)) > 0$. We claim that
$$
\lambda_{\epsilon} >  C_* := \min\Big\{ \frac{1}{\mathcal{K}_2(\delta'/4, K)}\Big(\frac{\delta'}{4\|e\|_\infty}\Big)^{p-1}, \frac{\delta'}{4}\Big\} .
$$
 In fact, otherwise by taking $k = \delta'/4$ in the second inequality in (\ref{8}),   we conclude that $(\lambda_\epsilon,  u_{\epsilon} ) \in B_{3\delta'/ 4}(0,0)\subset \mathbb{R}\times C(\overline{\Omega})$, which is an absurd as $(\lambda_{\epsilon}, u_\epsilon) \in \partial U. $

Now, by defining $\underline{u}_{\epsilon} = \lambda_{\epsilon}^{\frac{1}{p-1}}\mathcal{K}_1(K, U)\phi_1$, where $\mathcal{K}_1(K, U)$  will be chosen later, it follows from Picone's inequality, hypothesis $(A_0)$ and the fact that $(\lambda_\epsilon,u_\epsilon)$ is a solution of $(P_\epsilon)$, that
\begin{eqnarray}\label{9}
0 & \leq & \displaystyle\int_{\Omega}\Big[ |\nabla\underline{u}_{\epsilon}|^{p-2}\nabla\underline{u}_{\epsilon}\nabla \Big(\frac{(\underline{u}_{\epsilon} + \epsilon)^p - ({u}_{\epsilon} + \epsilon)^p}{(\underline{u}_{\epsilon} + \epsilon)^{p-1}}\Big)^+ -  |\nabla{u}_{\epsilon}|^{p-2}\nabla{u}_{\epsilon}\nabla \Big(\frac{(\underline{u}_{\epsilon} + \epsilon)^p - ({u}_{\epsilon} + \epsilon)^p}{({u}_{\epsilon} + \epsilon)^{p-1}}\Big)^+\Big]^+dx \nonumber\\
& \leq & \lambda_{\epsilon} \displaystyle\int_{\Omega} \Big[\frac{\lambda_1\mathcal{K}_1^{p-1}\phi_1^{p-1}}{(\lambda_{\epsilon}^{1/(p-1)}\mathcal{K}_1\phi_1+ \epsilon)^{p-1}} - \frac{f(x, u_\epsilon + \epsilon)}{(u_\epsilon + \epsilon)^{p-1}A_K}\Big]\Big((\underline{u}_{\epsilon} + \epsilon)^p - (u_{\epsilon} + \epsilon)^p\Big)^+ dx \nonumber\\
& \leq & \lambda_{\epsilon}\displaystyle\int_{\Omega}\Big[ \frac{\lambda_1}{\lambda_{\epsilon}} - \frac{f(x, u_\epsilon + \epsilon)}{(u_{\epsilon} + \epsilon)^{p-1}\displaystyle A_K}\Big]\Big((\underline{u}_{\epsilon} + \epsilon)^p - (u_{\epsilon} + \epsilon)^p\Big)^+ dx,
\end{eqnarray}
where $A_K = \max_{\overline{\Omega}\times [0, K^{\gamma}|\Omega|]} A$.

To complete the proof, let us argument by contradiction. First,  let us  fix $\tilde{K} > \big( \lambda_1 A_K\big) /C_*$ and conclude from hypothesis $(f_0)$  that there exists  $a > 0$ small enough such that $f(x, t) \geq \tilde{K}t^{p-1},$ for all $x\in \Omega$ and $0 < t < a$. Hence, by choosing  $\mathcal{K}_1(K, U) = {a}/\Big({4K^{\frac{1}{p-1}}\|\phi_1\|_{\infty}}\Big)$, we claim that $[\underline{u}_{\epsilon} > u_\epsilon]$ has zero measure for every $ \epsilon <\epsilon_0:=a/4 $ given. Otherwise, if we assume $\vert [\underline{u}_{\epsilon} > u_\epsilon] \vert >0$  for some $0<\epsilon <\epsilon_0 $, we get
$${u}_{\epsilon} + \epsilon \leq \underline{u}_{\epsilon} + \epsilon < \frac{a}{2} ~\mbox{on } [\underline{u}_{\epsilon} > u_\epsilon].$$

Therefore, by going back to (\ref{9}) and using ${\lambda_1}/{\lambda_\epsilon} \leq {\lambda_1}/{C_*}$, we have
\begin{eqnarray*} 0 &\leq & \lambda_{\epsilon}\displaystyle\int_{\Omega}\Big[ \frac{\lambda_1}{\lambda_{\epsilon}} - \frac{f(x, u_\epsilon + \epsilon)}{(u_{\epsilon} + \epsilon)^{p-1}A_K}\Big]\Big((\underline{u}_{\epsilon} + \epsilon)^p - (u_{\epsilon} + \epsilon)^p\Big)^+ dx \\
&  \leq & \lambda_{\epsilon}\displaystyle\int_{\Omega}\Big[ \frac{\lambda_1}{C_*} - \frac{\tilde{K}(u_{\epsilon}+\epsilon)^{p-1}}{(u_{\epsilon} + \epsilon)^{p-1}A_K}\Big]\Big((\underline{u}_{\epsilon} + \epsilon)^p - (u_{\epsilon} + \epsilon)^p\Big)^+ dx < 0,
\end{eqnarray*}
which is an absurd. Hence, $\lambda_{\epsilon}^{\frac{1}{p-1}}\mathcal{K}_1(K,U) \phi_1 \leq u_{\epsilon}$ in $\Omega$ for all $0 < \epsilon < \epsilon_0$, as we claimed.

\end{proof}
\fim
\vspace{0.2cm}

\noindent\textit{\textbf{Completion of proof of Theorem \ref{T1}:}}

For each $i \in \mathbb{N}$ given, define $$\mathcal{F}_i = \Big\{(\lambda, u) \in \mathbb{R}^+\times C(\overline{\Omega}) ~\mbox{that solves (P)} :  \frac{\lambda^{\frac{1}{p-1}}}{i} \phi_1(x) \leq u(x) \leq k + \lambda^{\frac{1}{p-1}}\mathcal{K}_2(k, i)^{\frac{1}{p-1}}e(x) ~ \mbox{in} ~ \Omega ~\mbox{for each} ~ k \in (0, i] \Big\},$$
 where  $\mathcal{K}_2(k,i)$ was introduced in the Lemma \ref{L2}.

To end the proof, it suffices to set
\begin{equation}
\label{cd}
\mathcal{F} = \displaystyle\bigcup_{i \in \mathbb{N}} \mathcal{F}_i \cup \{(0,0)\}\subset \mathbb{R}^+ \times C(\overline{\Omega})
\end{equation}
and  prove that there is an unbounded connected component $\Sigma \subset \mathcal{F}$. By standard argument of  Topology \cite{MR0099642}, the existence of  $\Sigma$  is a consequence of the following two claims:
\vspace{0.1cm}

\noindent \textit{Claim 1:} For each $U \subset \mathbb{R} \times C(\overline{\Omega}) $  bounded neighborhood
of $(0, 0)$ in $ \mathbb{R} \times C(\overline{\Omega})$,  there is a solution $(\lambda, u) \in \partial U \cap \mathcal{F}. $
\vspace{0.1cm}

\noindent \textit{Claim 2:} Closed and bounded (in $\mathbb{R}\times C(\overline{\Omega}))$  subsets of $\mathcal{F}$ are compact.
\vspace{0.1cm}

Let us prove each of the above claims one by one.
\vspace{0.1cm}

\textit{Proof of Claim 1:} Consider $U \subset \mathbb{R} \times C(\overline{\Omega}) $ be a  bounded neighborhood of $(0, 0)$ in $ \mathbb{R} \times C(\overline{\Omega})$ and a sequence $\epsilon_n \rightarrow 0^+$. By the Lemma \ref{L1},  there  exists $(\lambda_{n}, u_{n}) =(\lambda_{\epsilon_n}, u_{\epsilon_n}) \in \partial U \cap \Big( (0, \infty)\times W_0^{1,p}(\Omega) \Big)$ a solution of $(P_{\epsilon_n}),$ for each $n \in \mathbb{N}$.  Moreover, as $U$ is a bounded set, we can find a positive constant $K > 0$ such that $0\leq \lambda_n \leq K$ and $0\leq u_n \leq K$ in $\Omega$. Thus, by the Lemma \ref{L2}, we obtain
\begin{equation}\label{10}
\lambda_n^{\frac{1}{p-1}}\mathcal{K}_1(K,U) \phi_1 \leq u_n \leq k + \lambda_{n}^{\frac{1}{p-1}}\mathcal{K}_2(k, K)^{\frac{1}{p-1}}e ~~ \mbox{in} ~\Omega,
\end{equation}
for all $n \in \mathbb{N}$ sufficiently large and for each $k \in (0, K]$ given.

Suppose that $\lambda_n \rightarrow \lambda \geq 0$. If $\lambda = 0$, we conclude by (\ref{10}) that $u_n \rightarrow 0$ in $C(\overline{\Omega})$, that is, $(\lambda_n, u_n) \rightarrow (0,0)$ in $\mathbb{R}\times C(\overline{\Omega})$. Since $(\lambda_n, u_n) \in \partial U$ and $U$ is a bounded neighborhood of $(0,0)$, we obtain a contradiction. Therefore $\lambda > 0$, which implies that $0 < \lambda - \delta' < \lambda_n < \lambda + \delta'$ for $n$ sufficiently large and some $\delta' > 0$.

Consider a sequence $(\Omega_l)$ of open sets in $\Omega$ such that $\Omega_l \subset \Omega_{l+1}$ and $\bigcup_l \Omega_l = \Omega$ and define $\delta_l = \displaystyle\min_{\overline{\Omega}_l}(\lambda - \delta')^{\frac{1}{p-1}}\mathcal{K}_1(K,U) \phi_1,$ for each $l \in \mathbb{N}$. Taking $\varphi = (u_n - \delta_1)^+$ as a test function in $(P_{\epsilon_n})$, using (\ref{10}) and the hypothesis $(A_0)$, we obtain
\begin{eqnarray*}
\displaystyle\int_{[u_n \geq \delta_1]}|\nabla u_n|^pdx  =  \lambda_n\displaystyle\int_{[u_n \geq \delta_1]}\frac{f(x, u_n + \epsilon_n)}{A\Big(x, \int_{\Omega}u_n^{\gamma}\Big)}(u_n - \delta_1)^+dx \leq C_1,
\end{eqnarray*}
where $C_1 > 0 $ is a real constant independent of $n$. Thus, it follows from the previous inequality that $\{u_n\}$ is bounded in $W^{1,p}(\Omega_1). $ Hence, there exists $u_{\Omega_1} \in W^{1,p}(\Omega_1)$ and a subsequence $\{u_{n_j^1}\}$ of $\{u_n\}$ such that
$$
 \left\{
\begin{array}{l}
u_{n_j^1} \rightharpoonup u_{\Omega_1}  \ \mbox{weakly in}  \ W^{1,p}(\Omega_1) \  \mbox{and strongly in} \ L^q(\Omega_1)  \ \mbox{for} \ 1 \leq q < p^* \\
u_{n_j^1} \rightarrow u_{\Omega_1} \ \
a.e.  \ \mbox{in } \ \Omega_1.
\end{array}
\right.
$$

Proceeding as above, we can obtain subsequences $\{u_{n_j^l}\}$ of $\{u_n\}$, with $\{u_{n_j^{l + 1}}\} \subset \{u_{n_j^{l}}\}$, and functions $ u_{\Omega_l} \in W^{1,p}(\Omega_l)$ such that
$$
 \left\{
\begin{array}{l}
u_{n_j^l} \rightharpoonup u_{\Omega_l},\ \mbox{weakly in }  \ W^{1,p}(\Omega_l) \  \mbox{and strongly in} \ L^p(\Omega_l) \ \mbox{for} \ 1 \leq q < p^*  \\
u_{n_j^l} \rightarrow u_{\Omega_l} \ \
a.e. \ \mbox{in}   \ \Omega_l.
\end{array}
\right.
$$

By construction,  we have that $ u_{\Omega_{l+1}}\Big|_{\Omega_l} = u_{\Omega_l}. $ Hence, by defining $$u = \left\{
\begin{array}{l} u_{\Omega_1} \ \ \mbox{in} \ \ \Omega_1, \\ u_{\Omega_{l+1}} \ \ \mbox{in} \ \ \Omega_{l+1}\backslash \Omega_l
, \end{array}
\right.$$ we have that  $u \in  W_{\mathrm{loc}}^{1,p}(\Omega)$ and satisfies (\ref{10}).
In particular, by choosing $i>K$ large enough and using that $\mathcal{K}_2(k,\cdot)$ is non-decreasing, we have that
\begin{equation}\label{11}
\frac{\lambda^{\frac{1}{p-1}}}{i} \phi_1(x) \leq u(x) \leq k + \lambda^{\frac{1}{p-1}}\mathcal{K}_2(k, i)^{\frac{1}{p-1}}e(x)
\end{equation}
holds for each $k \in (0,i]$.

Furthermore, we claim that $(\lambda,u)$ is a solution for $(P)$. Indeed, by taking $\varphi \in C_c^{\infty}(\Omega)$ and using Theorem 2.1 in  \cite{MR1183665}, we have
\begin{equation}\label{12}
\displaystyle\int_{\Omega}|\nabla u_n|^{p-2}\nabla u_n\nabla \varphi dx \rightarrow \displaystyle\int_{\Omega}|\nabla u|^{p-2}\nabla u\nabla \varphi dx,
\end{equation}
up to a subsequence. On the other side, by using the continuity of $f$, the inequality (\ref{10}) and the hypothesis $(A_0)$, we obtain from Lebesgue Dominated Convergence Theorem that
 \begin{equation}\label{13}
\lambda_n\displaystyle\int_{\Omega}\frac{f(x, u_n + \epsilon_n)}{A\Big(x, \int_{\Omega}u_n^{\gamma}\Big)}\varphi dx \rightarrow \lambda\displaystyle\int_{\Omega}\frac{f(x, u )}{A\Big(x, \int_{\Omega}u^{\gamma}\Big)}\varphi dx.
\end{equation}
Thus, from (\ref{12}) and (\ref{13}) it is evident that $(\lambda, u)$ satisfies (\ref{4}). Also, by (\ref{11}) we obtain that $u > 0$ (in the sense of Definition \ref{D2}).  To verify that $u$ satisfies the boundary condition (see Definition \ref{D1}), it sufficient to note that the arguments used above lead us to the fact that the sequence $(u_n - \epsilon)^+$ is  bounded in $W_0^{1,p}(\Omega)$ as well. Therefore, $(u -\epsilon)^+ \in W_0^{1,p}(\Omega)$ for each $\epsilon > 0$ given.

 Finally,  by the continuity of $f$, hypothesis $(A_0)$ and (\ref{10}), we obtain from the classical regularity arguments that $u \in C(\Omega)$ and $u_n \rightarrow u$ in $C(\Theta)$, for each compact set $\Theta \subset \Omega$ given. Thus, by using this fact and (\ref{10}), we obtain that $(\lambda_n, u_n) \rightarrow (\lambda, u)$ in $\mathbb{R}\times C(\overline{\Omega})$, which on combining with (\ref{11}) implies that $(\lambda, u) \in  \partial U \cap \mathcal{F}_i \subset  \partial U \cap \mathcal{F}$, as required.
\vspace{0.2cm}

\noindent\textit{Proof of Claim 2:} Let $\{(\lambda_n, u_n)\} \subset \mathcal{F}$ be a bounded sequence (in $\mathbb{R}\times C(\overline{\Omega}))$. We aim to prove that $\{(\lambda_n, u_n)\}$ admits a subsequence that converges to some element of $\mathcal{F}. $

Initially, let us suppose that finitely many terms of $\{(\lambda_n, u_n)\}$ belongs to
 $\mathbb{R} \times C(\overline{\Omega}) \backslash B_{\delta'}(0,0), $ for each $ \delta'> 0 $ given.  In this case, $(0,0)$ would be an accumulation point of the sequence and our claim will hold. Otherwise,  let us assume that infinitely many terms of $ \{(\lambda_n, u_n) \}$ belongs to  $\mathbb{R} \times C(\overline{\Omega}) \backslash B_{\delta'}(0,0)  ,$ for some $\delta'>0$. Since $ \{(\lambda_n, u_n) \}$ is bounded by a constant $K>0$, the second inequality in (\ref{8}) is true.  Apart from this, since $\Vert (\lambda_n,u_n) \Vert_{\mathbb{R}\times C(\overline{\Omega})} \geq \delta'$ (just for  the subsequence in our assumption), the first inequality in (\ref{8}) holds true as well. Hence, by  fixing $i \in \mathbb{N}$ sufficiently large, we get that $ \{(\lambda_n, u_n) \} \subset \mathcal{F}_i$ for that subsequence.

 Let us fix such subsequence. By the boundedness of $\{\lambda_n \}\subset \mathbb{R}$ and $(\lambda_n, u_n) \subset \mathcal{F}_i \cap \Big(\big(\mathbb{R} \times C(\overline{\Omega}) \big)\backslash B_{\delta'}(0,0)\Big) $, it follows that $\lambda_n \rightarrow \lambda > 0$, up to subsequence.  As a consequence of this, we get
 \begin{equation}
 \label{117}
 \frac{\lambda^{1/(p-1)}}{2i} \phi_1 \leq u_n \leq K~\mbox{in  }\Omega
  \end{equation}
for $n \in \mathbb{N}$ large enough.

 Let $U \subset \subset \Omega$ and $\varphi \in C_c^{\infty}(\Omega)$ such that $0 \leq \varphi \leq 1,  ~\varphi = 1 $ in $U $ with $ U \subset \Theta := \mbox{supp}~ \varphi $. Thus, by (\ref{117}), we have a uniform bound of $(f(x,u_n))$ on $\Theta \times [k,K]$, where $k:=\min_{\Theta}\frac{\lambda^{1/(p-1)}}{2i} \phi_1>0$. Hence, using this information together with boundedness of $(\lambda_n,u_n)$ in $\mathbb{R}\times C(\overline{\Omega}))$, H\"older's inequality  and the hypothesis $(A_0)$, we have
\begin{eqnarray*}
& & \frac{1}{2^p}\displaystyle\int_{\Theta}|\nabla(\varphi u_n)|^p dx  =  \frac{1}{2^p}\displaystyle\int_{\Theta}|\nabla \varphi u_n + \nabla u_n \varphi|^p dx  \leq  \displaystyle\int_{\Theta}|\nabla \varphi|^p{u_n}^pdx + \displaystyle\int_{\Theta} |\nabla u_n|^p\varphi^pdx  \\
& & \leq  C_1\displaystyle\int_{\Theta} |\nabla \varphi|^pdx + \displaystyle\int_{\Theta}|\nabla u_n|^{p-2}\nabla u_n \nabla u_n \varphi^pdx  - \displaystyle\int_{\Theta} |\nabla u_n|^{p-2}\nabla u_n \nabla\varphi(p\varphi^{p-1}u_n)dx \\
&  & \leq C_1 \displaystyle\int_{\Theta}|\nabla \varphi|^pdx + \lambda_n\displaystyle\int_{\Theta} \frac{f(x,u_n)u_n}{A\Big(x, \displaystyle\int_{\Omega}u_n^{\gamma}\Big) }\varphi^pdx + C_2\displaystyle\int_{\Theta}|\nabla u_n|^{p-1}|\nabla \varphi|\varphi^{p-1}u_n dx  \\
& & {\leq}  C_3\Big[1 + \Big(\displaystyle\int_{\Theta}|\varphi\nabla u_n|^pdx\Big)^{\frac{p-1}{p}} \Big(\displaystyle\int_{\Theta} (|u_n\nabla \varphi |)^pdx\Big)^{\frac{1}{p}} \Big] ~~~~~~~(\mbox{using} ~ {(A_0)})\\
& &\leq  C_4 \Big[1 + \Big(\displaystyle\int_\Theta |\nabla (\varphi u_n)|^p dx\Big)^{\frac{p-1}{p}}\Big],
\end{eqnarray*}
where $C_4$ is a positive constant, independent of $n$. Thus, $\{\varphi u_n\}$ is  bounded in $W_0^{1,p}(\Theta)$ and as a consequence of this, $\{u_n\}$ is bounded in $W^{1,p}(U)$. By using the arbitrariness of $U$ and proceeding as in the proof of the Claim 1, we obtain  a function $u \in W_{\mathrm{loc}}^{1,p}(\Omega) \cap C(\overline{\Omega})$ such that
\begin{equation}
\label{1180}
 \left\{
\begin{array}{l}
u_{n} \rightharpoonup u  \ \mbox{weakly in } \ \ W^{1,p}(U) \ \ \mbox{for each  } \ U \subset \subset \Omega, \\
  u_n \rightarrow u ~\mbox{in} ~ C(\overline{\Omega}), \\
  \frac{\lambda^{\frac{1}{p-1}}}{i} \phi_1(x) \leq u(x) \leq k + \lambda^{\frac{1}{p-1}}\mathcal{K}_2(k, i)^{\frac{1}{p-1}}e(x) ~~ \mbox{in} ~ \Omega  ~ \mbox{for all} ~ k \in (0, i]
\end{array}
\right.
\end{equation}
for $i$ as fixed before.

From the last inequality in (\ref{1180}), it follows that $(u-\epsilon)^+ \in W_0^{1,p}(\Omega)$ for each $\epsilon>0$ given, as noted in Claim 1. Hence, to complete the proof of the existence of the  \textit{continuum}, we just need to show that $ (\lambda, u) $ satisfies the equation in $(P)$, that is, (\ref{4}). Since $(\lambda_n, u_n)$ solves $(P_{\epsilon_n})$, it follows from density arguments, (\ref{117}) and (\ref{1180}) that
\begin{equation}\label{14}
\displaystyle\int_{\Omega}|\nabla u_n|^{p-2}\nabla u_n \nabla \Big(\varphi(u_n  - u)\Big)dx  = \lambda_n\displaystyle\int_{\Omega}\frac{f(x, u_n)}{A\Big(x, \int_{\Omega}u_n^{\gamma}\Big)}\varphi(u_n - u)dx \rightarrow 0
\end{equation}
for all $ \varphi \in C_c^{\infty}(\Omega)$.

Since $\{u_n\}$ is a bounded sequence in $W_{\mathrm{loc}}^{1,p}(\Omega)$, we obtain
\begin{equation}\label{15}
\Big| \displaystyle\int_\Omega  |\nabla u_n|^{p-2}\nabla u_n \nabla \varphi (u_n - u)dx \Big| \leq C\|u_n - u\|_{p} \rightarrow 0
\end{equation}
by using the H\"older's inequality. Therefore, it  follows from (\ref{14}) and (\ref{15}) that
$$ \displaystyle\int_{\Omega} \varphi\Big(|\nabla u_n|^{p-2}\nabla u_n - |\nabla u|^{p-2}\nabla u\Big)\nabla (u_n - u)dx \rightarrow 0, $$ up to subsequence, which implies that $\nabla u_n \rightarrow \nabla u$  $a.e. $ in $\Omega. $

Thus,  proceeding as in proof of the Claim 1, we obtain that $(\lambda, u) \in \mathcal{F}_i\subset \mathcal{F}$, which concludes the proof of the existence of an unbounded \textit{continuum} of positive solutions for $(P)$.

In order to finish the proof of later part of the Theorem \ref{T1}, let us assume $(f_\infty)$  and  $A(x,t) > a_0  ~\mbox{ in} ~\overline{\Omega}\times \mathbb{R}^+$ holds for some $a_0>0$.
Assume by contradiction that $Proj_{\mathbb{R}} \Sigma \subset [0,\lambda^*]$ for some $0 < \lambda^* < \infty$, that is,    $0\leq \lambda \leq \lambda^* $ whenever  $(\lambda, u) \in \Sigma$. Hence, by taking $R> 0$ and $\epsilon_n = 1/n $ ($n \in \mathbb{N}$), we obtain by  Lemma \ref{L1} that there exists $(\lambda_n, u_n) =(\lambda_{n,R}, u_{n,R})\in \Sigma_n\cap \partial B_R(0,0)$, where $\Sigma_n$ is  the unbounded $\epsilon_n$-\textit{continuum} of positive solutions of $(P_{\epsilon_n})$ .

We claim that there exists $R_0 > 0$ such that $\lambda_n \geq \lambda^* + 1 $ for all $n \in \mathbb{N}$ and $R > R_0$. Otherwise, we can find a sequence $R_l \rightarrow \infty$ and a subsequence $\{u_{n_l}\}$ satisfying
\begin{equation}
\label{120}
\|u_{n_l}\|_{\infty} = R_l - \lambda_{n_l} \geq R_l - \lambda^* - 1.
\end{equation}
 However, by  Lemma \ref{L2} we have $\|u_{n_l}\|_{\infty} \leq 1 + \mathcal{K}_2(1, R_l)^{1/(p-1)}(\lambda^* + 1)^{1/(p-1)}\|e\|_{\infty}, $ where $\mathcal{K}_2(1, R_l) = \max\Big\{\frac{f(x,t)}{a_{R_l}}: x \in \overline{\Omega} ~\mbox{and}~ 1 \leq t \leq R_l+1\Big\}$ with $a_{R_l} = \displaystyle\min_{\overline{\Omega}\times [0, R_l^{\gamma}|\Omega|]} A\geq a_0$ by our assumption.  Hence, it follows from the hypothesis $(f_\infty)$ that for each $\epsilon > 0$ there exists a positive constant $C_\epsilon^1$ such that $\mathcal{K}_2(1, R_l) \leq C^1_{\epsilon} + \frac{\epsilon}{a_0} R_l^{p-1}$ holds for all $l \in \mathbb{N}$ sufficiently large. As a consequence of these information, we obtain
\begin{equation}\label{Z1}
\|u_{n_l}\|_{\infty} \leq 1 + \Big(C^1_\epsilon + \frac{\epsilon}{a_0} R_l^{p-1}\Big)^{1/(p-1)}(\lambda^* + 1)^{1/(p-1)}\|e\|_{\infty} \leq C^2_\epsilon + C_2 \epsilon^{1/(p-1)}R_l,
\end{equation}
for $l$ large enough and for some positive constants $C^2_\epsilon$ and $C_2$, where $C_2$ is independent of $\epsilon$.

Let   $\epsilon > 0$ be such that  $1 - \epsilon^{1/(p-1)}C_2 > 0$. Since $R_l \to \infty$, we can take a $l$ large enough such that
$R_l > {C_2^{\epsilon+ \lambda^* +1}}/{(1 - \epsilon^{1/(p-1)}C_2)}.$
Thus, by going back to (\ref{Z1	}), we obtain for such $l$ that  $\|u_{n_l}\|_{\infty} \leq C^2_\epsilon + C_2\epsilon^{1/(p-1)}R_l < R_l - \lambda^* - 1 $ holds, but this is a contradiction by (\ref{120}).

Therefore, by fixing  $R> R_0 > 0$ and proceeding as in the proof of the Claim 1, we obtain that $(\lambda_n, u_n)=(\lambda_{n,R}, u_{n,R})$ converges in $\mathbb{R}\times C(\overline{\Omega}) $  to a pair $(\lambda, u) \in \Sigma \cap \partial B_R(0,0)$, which implies that $\lambda \geq \lambda^* + 1$, but this is not possible by the contrary hypothesis of $Proj_{\mathbb{R}^+} \Sigma \subset [0, \lambda^*]$. This ends the proof.
\fim

\section{  $W_{\mathrm{loc}}^{1,p}(\Omega)$-behavior to a parameter for $(p-1)$-sublinear problems}

Let us present some  results which are important in itself and are required to overcome some obstacles on the strategies of Rabinowitz \cite{MR0301587} and Figueiredo-Sousa \cite{MR3694626}, in order to approach non-autonomous non-local singular problems involving p-Laplacian operator in the setting of $W_{\mathrm{loc}}^{1,p}(\Omega)$-solutions.

To enunciate the first one, let us define a  subsolution and a supersolution for the problem
\begin{equation}\label{lsp} \left\{
\begin{array}{l}
  -\Delta_p u=  a_1(x) u^{\theta_1} + a_2(x)u^{\theta_2}  ~\mbox{in } \Omega,\\
    u>0    ~\mbox{in } \partial\Omega,~~
    u>0    ~\mbox{on }  \Omega,
 \end{array}
\right.
\end{equation}
in the following sense.
\begin{definition} A function $\underline{v} \in W_{\mathrm{loc}}^{1,p}(\Omega) $ is a subsolution of $(\ref{lsp})$ if:
\begin{itemize}
\item[$i)$]  there is a positive constant $c_\Theta$ such that $\underline{v} \geq c_\Theta$ in $\Theta$ for each $\Theta \subset \subset \Omega$ given;
\item[$ii)$] the inequality
\begin{eqnarray}\label{3.9}
 \displaystyle\int_{\Omega} |\nabla \underline{v}|^{p-2}\nabla \underline{v} \nabla \varphi dx \leq   \displaystyle\int_{\Omega} \Big({a_1(x) }{\underline{v}^{\theta_1}} + a_2(x)\underline{v}^{\theta_2}\Big) \varphi dx
 \end{eqnarray}
 holds for all $ 0 \leq \varphi \in C_{c}^{\infty}(\Omega)$.
 When $\overline{v}\in W_{\mathrm{loc}}^{1,p}(\Omega) $ satisfies  the reversed inequality in $(\ref{3.9})$, it is called a supersolution of $(\ref{lsp})$.
\end{itemize}
\end{definition}

In this context, we  state a Comparison Principle for $W_{\mathrm{loc}}^{1,p}(\Omega)$-sub and supersolutions, proved in Theorem 2.1 of \cite{NOSSO}.

\begin{thmlet}[$W_{\mathrm{loc}}^{1,p}(\Omega)$-Comparison Principle]\label{unicidade1} Suppose that  $- \infty < \theta_1, \theta_2 < p-1$ and $a_1 + a_2>0$ in $\Omega$ hold. Assume that the pair $(\theta_i, a_i)$ satisfies one of the following hypotheses:
\begin{itemize}
\item [$(h)_1$:] $-1 < \theta_i < p-1 $ and $a_i \in L^{(\frac{p^*}{p^* -1-\theta_i})}(\Omega)$,
\item [$(h)_2$:]$\theta_i < -1$ and $a_i \in L^1(\Omega),$
\item [$(h)_3$:]$\theta_i = -1$ and $a_i \in L^s(\Omega)$ for some $s>1$
\end{itemize}
for $i \in \{1,2\}$.
 If $\underline{v}, \overline{v}  \in W_{{\mathrm{loc}}}^{1,p}(\Omega)$ are subsolution and supersolution of $(\ref{lsp})$, respectively, with  $\underline{v} \leq 0$ in $\partial \Omega$, then $\underline{v} \leq \overline{v}$ a.e. in $\Omega$.
\end{thmlet}

Following the proof of the above Theorem, we have the next result.
\begin{corollary}
\label{R1}
Assume that the same assumptions of the Theorem \ref{unicidade1}, $a_2 \leq a_1$ in $\Omega$ and $\theta_1 \leq \theta_2$ hold. If $\underline{v}, \overline{v}  \in W_{{\mathrm{loc}}}^{1,p}(\Omega)$ are subsolution and supersolution of
\[ \left\{
\begin{array}{l}
  -\Delta_p u=  a_1(x)u^{\theta_1}\chi_{[u < a]} + a_2(x)u^{\theta_2}\chi_{[u \geq a]} ~\mbox{in } \Omega,\\
    u>0    ~\mbox{in } \partial\Omega,~~
    u>0    ~\mbox{on }  \Omega,
 \end{array}
\right.
\]
respectively, with $\underline{v} \leq 0$ in $\partial \Omega$ and $0\leq a <1$, then $\underline{v} \leq \overline{v}$ a.e. in $\Omega$.
\end{corollary}
\begin{proof}
It is sufficient to revisit the proof of Theorem \ref{unicidade1} and observe that, under the contradictory assumption $\vert[(u^p - v^p)^+\phi > 0]\vert > 0$, we also obtain
$$\displaystyle\int_{[u \geq v]}\Big[\frac{a_1(x)u^{\theta_1}\chi_{[u < a]} + a_2(x)u^{\theta_2}\chi_{[u \geq a]}}{u^{p-1}} -  \frac{a_1(x)v^{\theta_1}\chi_{[v < a]} + a_2(x)v^{\theta_2}\chi_{[v \geq a]}}{v^{p-1}}\Big](u^p - v^p)\varphi dx < 0, $$ which leads us to a similar contradiction, as in the proof of Theorem \ref{unicidade1}.
 \fim
\end{proof}

The next Lemma brings out an important parametric behavior of the solution of $(p-1)-$sublinear problem. This result is crucial in our approach.

\begin{lemma}\label{L3}
Assume that $(f_1)$ and $(f_2)$ are satisfied with $c_0, c_\infty > 0$ in $\overline{\Omega}$ and $\delta \leq \beta$. Then, there exist $\alpha_0,  \alpha_{\infty},m_1,m_2 > 0$  such that  any positive solution $u \in W_{\mathrm{loc}}^{1,p}(\Omega) $ of
\begin{equation}\label{lema3}
-\Delta_p u = \alpha f(x,u) ~~\mbox{in} ~ \Omega, ~~ u|_{\partial\Omega} = 0,
\end{equation}
(see definition $\ref{D2}$ with $A\equiv 1$) satisfies
\begin{equation}\label{SS}
\alpha^{\tau}m_1\phi_1 \leq u \leq \alpha^{\tau}m_2e^t~\mbox{in }\overline{\Omega},
\end{equation}
where $t = \min\{1, (p-1)/(p-1- \delta)\} $,
$$a) ~\tau = {1/(p-1-\delta)} ~\mbox{ for all} ~\alpha \in (0, \alpha_0) ~~~~\mbox{and} ~~~~~ b)~\tau = {1}/{(p-1-\beta)} ~~\mbox{for all} ~\alpha > \alpha_{\infty}. $$
\end{lemma}
\begin{proof} Let $u \in W_{\mathrm{loc}}^{1,p}(\Omega)\cap C(\overline{\Omega}) $ be a solution of (\ref{lema3}). It follows from $(f_1)$ and $(f_2)$ that there exist $m, M > 0$ such that
$$
m\Big(u^{\delta}\chi_{[u < a]} + u^{\beta}\chi_{[u \geq a]} \Big) \leq f(x,u) \le M\Big(u^{\delta} + u^{\beta}\Big)
$$
holds for some  $0 < a <1$ small enough, that is,  $u \in W_{\mathrm{loc}}^{1,p}(\Omega)\cap C(\overline{\Omega})$ is a subsolution for
\begin{equation}\label{sub}
-\Delta_p u = \alpha M \Big(u^{\delta} + u^{\beta}\Big)
\end{equation}
and a supersolution for
\begin{equation}\label{super}
-\Delta_p u = \alpha m \Big(u^{\delta}\chi_{[u < a]} + u^{\beta}\chi_{[u \geq a]}\Big).
\end{equation}

Now, we build  a positive  supersolution for (\ref{sub}) and a positive subsolution for (\ref{super}), as required by Theorem \ref{unicidade1}. First, let us define $\overline{u}_\alpha = m_2\alpha^\tau e^t$,  $\alpha>0$, with $t = \min\{1, {(p-1)}/{(p-1-\delta)} \}$ and $\tau, m_2>0$ being constants independent of $\alpha$, to be chosen later. Thus, using that $0 < t \leq 1,$ we have
\begin{eqnarray*}
\displaystyle\int_{\Omega}|\nabla \overline{u}_{\alpha}|^{p-2}\nabla \overline{u}_{\alpha}\nabla \varphi dx {\geq} \displaystyle\int_{\Omega} |\nabla e|^{p-2}\nabla e\nabla\Big[\varphi(\alpha^\tau m_2e^{t-1}t)^{p-1}\Big]dx =  \displaystyle\int_{\Omega} \varphi(\alpha^\tau m_2e^{t-1}t)^{p-1}dx
\end{eqnarray*}
for each $0 \leq \varphi \in C_c^{\infty}(\Omega)$ given.

To verify that $ \overline{u}_{\alpha}$  is a supersolution for (\ref{sub}),  it is enough to show that
\begin{equation}\label{C}
(\alpha^\tau m_2t)^{p-1} \geq \alpha M\max\{1, \|e^{t(\beta -\delta)}\|_{\infty}\}\Big(m_2^{\delta}\alpha^{\tau\delta} + m_2^{\beta}\alpha^{\tau\beta}\Big)
\end{equation}
holds, for some appropriately chosen $\tau, m_2>0$.

To do this, let us fix $m_2 = \max\Big\{1, \Big(\frac{3M\max\{1, \|e^{t(\beta-\delta)}\|_\infty\}}{t^{p-1}}\Big)^{1/(p-1-\beta)}\Big\}$ and consider two cases on the size of $\alpha$. If $\alpha<1$,  we obtain that the inequality (\ref{C}) holds by choosing  $\tau = {1}/(p-1-\delta)$, while for $\alpha \geq 1$  we obtain  (\ref{C}) by taking $\tau = {1}/{(p-1- \beta)}$. Therefore, in both cases $\overline{u}_{\alpha}$ is a supersolution for (\ref{sub}) for every $\alpha >0$.

Next, we build a subsolution for (\ref{super}) as follows. Setting $ \underline{u}_{\alpha} = \alpha^\tau m_1 \phi_1 $, $\alpha>0$, we have that  $\underline{u}_{\alpha}$ will  be a subsolution for (\ref{super}) if
\begin{equation}\label{D}
 (m_1\alpha^\tau)^{(p-1)}\lambda_1\phi_1^{p-1} \leq \alpha m \Big(m_1^{\delta}\alpha^{\tau\delta}\phi_1^{\delta}\chi_{[m_1\alpha^\tau
 \phi_1 < a]} +m_1^{\beta}\alpha^{\tau\beta}\phi_1^{\beta}\chi_{[m_1\alpha^\tau\phi_1 \geq a]} \Big)
 \end{equation}
is satisfied, for some $\tau,m_1>0$ independent of $\alpha$.

Again, let us consider two cases on $\alpha$. First, let  $0 < \alpha <   \lambda_1 a^{p-1-\delta}/m$. By taking  $\tau= 1/(p-1 - \delta)$ and $m_1 = \Big(m/\lambda_1\|\phi_1^{1/\tau}\|_{\infty}\Big)^{\tau} = m^{1/(p-1-\delta)}/(\|\phi_1\|_\infty \lambda_1^{1/(p-1-\delta)}) $, the inequality (\ref {D}) holds. On the other hand, for $\alpha \geq \lambda_1 a^{p-1-\delta}/m$, let us take $\tau= 1/(p-1 - \beta)$ and $m_1 = \Big(m/\lambda_1\|\phi_1^{1/\tau}\|_{\infty}\Big)^{\tau} = m^{1/(p-1-\beta)}/(\|\phi_1\|_\infty \lambda_1^{1/(p-1-\beta)}) $ to obtain the inequality (\ref{D}) again. Therefore, in both cases, we have  that  $\underline{u}_\alpha$  is a subsolution of (\ref{super}) for each $\alpha>0 $ given.

Fix
$$\alpha_0 = \min\Big\{ 1, \frac{\lambda_1a^{p-1-\delta}}{m}\Big\} ~~\mbox{and} ~~\alpha_\infty = \max\Big\{1, \frac{\lambda_1a^{p-1-\delta}}{m}\Big\}.$$
Now, using $u$ as a subsolution of (\ref{sub}) and $\overline{u}_{\alpha} = \alpha^\tau m_2e^t$ as a supersolution of (\ref{sub}), for $\tau = 1/(p-1-\delta)$ and $\alpha < \alpha_0$, together with Theorem \ref{unicidade1}, we get the second inequality in the item$-a)$.

Moreover, using $u$ as a supersolution of (\ref{super}) and $\underline{u}_\alpha = \alpha^\tau m_1\phi_1$ as a subsolution of (\ref{super}), for $\tau = 1/(p-1-\delta)$ and $\alpha < \alpha_0$, together with Corollary \ref{R1}, we get the first inequality in item$-a)$.

Similarly, for $\alpha > \alpha_\infty$ and $\tau = 1/(p-1-\beta)$, arguing as before we get the both inequalities in item$-b)$.

\fim
\end{proof}

As immediate consequence of the proof of the previous Lemma, we have the following Corollary.
\begin{corollary} \label{C1}
Assume that $-\infty < \delta \leq   \beta <p-1$. If there exist $M,m>0$ and $0<u,v \in W_{\mathrm{loc}}^{1,p}(\Omega)\cap C(\overline{\Omega})$ such that:
\begin{enumerate}
\item[$(i)$] the inequality
\begin{equation}\label{in1}
-\Delta_p u \leq \alpha M (u^{\delta} + u^{\beta}) ~\mbox{in} ~\Omega ~\mbox{and} ~u \leq 0 ~ \mbox{on} ~ \partial \Omega
\end{equation}
holds, then $u$ satisfies the second inequality in (\ref{SS}), for some $m_2$ independent of $\alpha>0$, where $\tau$ is given in the items $a)-b)$ of the Lemma \ref{L3}. In particular, if $u$ satisfies $-\Delta_p u \leq  {L} (u^{\delta} + u^{\beta})$  for some $L>0$ and $u \leq 0$ on $\partial \Omega$, then $\Vert u \Vert_{\infty} \leq C(L)$,
\item[$(ii)$] the inequality
\begin{equation}\label{in2}
-\Delta_p v \geq \alpha m(v^{\delta}\chi_{[v < a]} + v^{\beta}\chi_{[v \geq a]}) ~\mbox{in} ~\Omega
\end{equation}
holds for some $0 < a <1$, then $v$ satisfies the first inequality in (\ref{SS}), for some $m_1$ independent of $\alpha>0$, where $\tau$ is given in the items $a)-b)$  of the Lemma \ref{L3}.
\end{enumerate}
\end{corollary}
\begin{proof}
{It remains only to prove the particular case in item $i)$. Without loss of generality, we can assume that $L > \alpha_\infty$. Thus, by identifying $\alpha = L$ and $M =1$ in (\ref{in1}), it follows from the first part of the proof of the above Lemma that  $u \leq m_2L^{1/(p-1- \beta)}e^t$, where $m_2 = \max\Big\{1, \Big(\frac{3\max\{1, \|e^{t(\beta-\delta)}\|_\infty\}}{t^{p-1}}\}\Big)^{1/(p-1-\beta)}\Big\}$. Therefore, $\|u\|_{\infty}  \leq   m_2L^{1/(p-1- \beta)}\|e^t\|_\infty := C(L). $ }
\fim
\end{proof}
\medskip

\section{Proof of Theorem \ref{T2} and \ref{T4}}

In this section, we will prove Theorems \ref{T2} and \ref{T4}. We also prove an existence and non-existence result for the degenerate problem ( i.e. $A(x,0)=0$ in $\Omega$ ) in Theorem \ref{T5}. We begin with Theorem \ref{T2}.\\

\noindent {\textbf{\textit{Proof of Theorem \ref{T2}:}}} First, we note that under the hypotheses $(A_0)$ and $(f_2)$, we are able to apply Theorem $\ref{T1}$ to guarantee the existence of an unbounded \textit{continuum} $\Sigma$ of positive $W_{\mathrm{loc}}^{1,p}(\Omega)\cap C(\overline{\Omega})$-solutions for $(P)$.
\begin{itemize}
\item[a)]
 Let us prove just the case $\{$$\theta{\gamma} = p-1-\beta$ and $ (A'_\infty)$$\}$, because the other one is similar. Assume by contradiction that $\Sigma$ is horizontally bounded. Then, there exists a sequence $(\lambda_n, u_n) \subset \Sigma $ and  $0 < \lambda^* <\infty$  such that $\lambda_n \leq \lambda^* $ and $\|u_n\|_{\infty} \rightarrow \infty$.
We claim that $\int_{\Omega}u_n^{\gamma}dx \rightarrow \infty$. Otherwise, it would follow from   $(A_0)$, $(f_1)$ and $(f_2)$ that
  $$-\Delta_p u_n \leq L\Big(u_n^{\delta} + u_n^{\beta}\Big) $$ holds, up to a subsequence, for some $L> 0$ independent of $n$. Using this information  and  Corollary \ref{C1}$-i)$, we obtain $\|u_n\|_{\infty}  \leq C(L)$ but this is a contradiction with the fact that  $\|u_n\|_{\infty} \rightarrow \infty$.

Now, for $t=\displaystyle \min\{1, (p-1)/(p-1-\delta)\}$, fix $ m_2 \in  (0, ~\min\{1, (\int_{\Omega} e^{t\gamma}dx )^{-1/{\gamma}}\})$ and $C_1>0$ such that
\begin{equation}\label{AB}
\frac{\lambda^*}{C_1} \leq \frac{m_2^{p-1-\delta }t^{p-1}}{2\max\{1, \|e\|_{\infty}^{t|\beta - \delta|}\}}.
\end{equation}

First, we note that as a consequence of  $\int_{\Omega} u_n^{\gamma}dx \rightarrow \infty$ and the hypothesis ($A'_\infty$), for $n$ large we have $A\Big(x, \int_{\Omega} u_n^{\gamma}dx\Big)\Big(\int_{\Omega}u_n^{\gamma}dx \Big)^{\theta} \geq C_1 > 0 $ which leads us to
$$
-\Delta_p u_n = \frac{\lambda_n(\int_{\Omega}u_n^{\gamma}dx )^{\theta}f(x, u_n)}{A\Big(x, \int_{\Omega} u_n^{\gamma}dx\Big)\Big(\int_{\Omega}u_n^{\gamma}dx \Big)^{\theta}} \leq {\frac{\lambda^*}{C_1}}\tilde{\lambda}_n \Big(u_n^{\delta} + u_n^{\beta}\Big),
$$
 where $\tilde{\lambda}_n = \Big(\int_{\Omega}u_n^{\gamma}dx \Big)^{\theta} $.

Next, let us define $\overline{u}_n = m_2\tilde{\lambda}_n^{\tau}, $ with $\tau = (p-1-\beta)^{-1}$.  By  proceeding  as in the proof of Lemma \ref{L3}$-b)$ and using  (\ref{AB}), we have
\[
-\Delta_p\overline{u}_n \geq {\frac{\lambda^*}{C1}}\tilde{\lambda}_n \Big(\overline{u}_n^{\delta} + \overline{u}_n^{\beta}\Big)
\]
for $n$ sufficiently large.

Therefore, by Theorem \ref{unicidade1} we obtain  $u_n \leq m_2\Big(\int_{\Omega}u_n^{\gamma}\Big)^{\theta \tau}e^t $, which results in
$$\int_{\Omega}u_n^{\gamma}dx \leq \Big(\int_{\Omega}u_n^{\gamma}dx\Big)^{\theta \tau \gamma}m_2^{\gamma}\int_{\Omega}e^{t\gamma}dx.$$
As $\theta\gamma = p-1-\beta$, it follows from the previous inequality that  $1 \leq m_2^{\gamma}\int_{\Omega}e^{t\gamma}dx$, but this is a contradiction by our choice of  $m_2 <  (\int_{\Omega} e^{t\gamma}dx )^{-1/{\gamma}}$.

\item[b)] Assume that there exists a sequence $ (\lambda_n, u_n) $ of solutions of $(P)$ such that $\lambda_n \rightarrow \infty$. We claim that $ \int_{\Omega} u_n^{\gamma}dx \rightarrow \infty $. Otherwise, by the hypotheses  $(f_1)$ and $(f_2)$ there exist  constants $C_1>0$ and $ 0 < a < 1$ such that
\begin{equation}\label{fg}
-\Delta_p u_n \geq C_1\lambda_n \Big(u_n^{\delta}\chi_{[u_n < a]} + u_n^{\beta}\chi_{[u_n \geq a]}\Big)
\end{equation}
holds, up to a subsequence. Thus, we obtain from  (\ref{fg}) and  Corollary \ref{C1}$-ii)$ that $\lambda_n^\tau m_1\phi_1 \leq u_n$
 for some $m_1 > 0$ independent of $n$, $\tau = (p-1-\beta)^{-1}$ and $n$  large enough.  Hence, from this we get  $C \geq  \int_\Omega u_n^{\gamma}dx \geq \lambda_n^{\tau\gamma}\int_{\Omega}\phi_1^{\gamma}dx \rightarrow \infty, $ which is a contradiction.

From the above claim and the hypothesis $0 \leq a_\infty <\infty$ on $\overline{\Omega}$, we obtain  $$A\Big(x, \int_{\Omega}u_n^{\gamma}dx \Big)\Big(\int_{\Omega}u_n^{\gamma}dx\Big)^{\theta} \leq C_{2}$$ for some constant $C_{2}>0$  and, as a consequence of this, we have
\[
-\Delta_p u_n \geq C_{3}\lambda_n\Big(\int_\Omega u_n^{\gamma}dx\Big)^{\theta}\Big(u_n^{\delta}\chi_{[u_n < a]} +  u_n^{\beta}\chi_{[u_n \geq a]}\Big)
\]
 for some $C_3>0$ indenpendent of $n$.

Now, by taking  $m = C_3$ and $\alpha = \lambda_n\Big(\int_\Omega u_n^{\gamma}dx\Big)^{\theta}$ in (\ref{in2}), it follows from Corollary \ref{C1}$-ii)$ that $ \lambda_n^\tau\Big(\int_{\Omega} u_n^{\gamma}dx\Big)^{\tau\theta}m_1\phi_1 \leq u_n$, for some $m_1 > 0$ independent of $n$, $\tau = (p-1-\beta)^{-1}$ and  $n$ sufficiently large. Thus, we conclude that
$ \lambda_n^{\gamma \tau} \leq C_{4}\Big(\displaystyle\int_{\Omega}u_n^{\gamma}dx\Big)^{1- \tau\theta\gamma}  = C_4 $ for some $C_4 > 0$, where in the last equality we used $\tau \theta\gamma= 1$. But this is a contradiction, since $\gamma \tau > 0$ and $\lambda_n \rightarrow \infty. $

Below, let us prove the items $i) - ii)$.
\begin{itemize}
\item[$i)$] Assume that there exists a sequence $(\lambda_n, u_n) \subset \Sigma$ such that $\lambda_n \rightarrow 0 $ and $\|u_n\|_{\infty} \rightarrow \infty$. In the same way as
proved  in the item $(a) $ above, we get $\int_{\Omega}u_n^{\gamma}dx \to \infty$. Using this fact and the hypothesis  $a_\infty > 0$ in $\overline{\Omega}$, we obtain  \begin{equation}\label{Z}
-\Delta_p u_n \leq C_1\lambda_n \Big(\int_{\Omega}u_n^{\gamma}dx \Big)^{\theta}(u_n^\delta + u_n^\beta),
\end{equation}
which implies that $\lambda_n \Big(\int_{\Omega}u_n^{\gamma}dx \Big)^{\theta} \to \infty$. If not, we would have $ C_1\lambda_n \Big(\int_{\Omega}u_n^{\gamma}dx \Big)^{\theta} \leq C_2$ for some $C_2$ large, hence  by  Corollary \ref{C1}$-i)$ we get $\|u_n\|_\infty \leq C( C_2)$. However, this is  a contradiction because we are supposing that $\|u_n\|_{\infty} \rightarrow \infty$.

 Therefore, by taking $M = C_1$ and $\alpha = \lambda_n\Big(\int_\Omega u_n^{\gamma}dx\Big)^{\theta}$ in (\ref{in1}) and applying Corollary \ref{C1}$-i)$, we get $u_n \leq m_2\lambda_n^{\tau}\Big(\int_{\Omega}u_n^{\gamma}dx \Big)^{\tau\theta}e^t$ for some $m_2$ independent of $n$, $\tau = (p-1-\beta)^{-1}$ and $n$ large enough,  which  lead us to conclude that $1 = \Big(\int_{\Omega}u_n^{\gamma}dx \Big)^{1 - \tau\theta\gamma}\leq C\lambda_n^{\tau\gamma } \rightarrow 0$ by the choice of $\theta$. This is impossible.

\item[$ii)$] Assume that there exists a sequence $(\lambda_n, u_n) \subset \Sigma$ such that $\lambda_n \rightarrow \lambda^* > 0$ and $\|u_n\|_{\infty} \rightarrow \infty.  $ Then, by the same idea as used to prove the item (a) above, we have that $\int_\Omega u_n^{\gamma}dx \to \infty$. Thus, for a given  $\epsilon> 0$, we obtain from  the hypothesis $ a_\infty \equiv 0 $ that $0 < \lambda^*/2  < \lambda_n $ and $A\Big(x, \int_{\Omega} u_n^{\gamma}dx\Big)\Big(\int_{\Omega} u_n^{\gamma}dx\Big)^{\theta} < \epsilon$ for all $n $  large as much as necessary. From this we obtain  that $-\Delta_p u_n \geq \frac{\lambda^*C_1}{2\epsilon}\Big(\int_\Omega u_n^{\gamma}dx\Big)^{\theta}(u_n^{\delta}\chi_{[u_n \leq a]} + u_n^{\beta}\chi_{[u_n > a]}) $, for some $C_1$ independent of  $n$ and $\epsilon>0$.

Hence, taking $m=C_1$ and $\alpha = \frac{\lambda^*}{2\epsilon}\Big(\int_\Omega u_n^{\gamma}dx\Big)^{\theta}$ in (\ref{in2}),  we get by the Corollary \ref{C1}$-ii)$ that $\Big(\frac{\lambda^*}{2\epsilon}\Big)^\tau\Big(\int_{\Omega} u_n^{\gamma}dx\Big)^{\theta \tau}m_1\phi_1 \leq u_n$  for some $m_1$ independent of $n$,  $\tau = (p-1-\beta)^{-1}$ and $n$ large. As a consequence of this information and by $\theta\gamma \geq p-1-\beta$, we obtain
$1 \geq \Big(\int_{\Omega} u_n^{\gamma}dx\Big)^{1 - \tau\gamma\theta } \geq \frac{C}{\epsilon^\tau},  $ which is an absurd for $\epsilon>0$ small enough, as $C$ is independent of $\epsilon$.
\end{itemize}

\item[c)] Assume that there exists a pair $(\lambda_n, u_n)$ which solves $(P)$ with $\lambda_n \rightarrow 0^+$. Then it must occurs that $\int_{\Omega}u_n^{\gamma}dx \rightarrow \infty$, otherwise
\[
-\Delta_p u_n \leq C_{1}\lambda_n\Big(u_n^{\delta} + u_n^{\beta}\Big)
\]
holds, up to subsequence. By taking $M = C_1$ and $\alpha = \lambda_n$ in (\ref{in1}), we get by Corollary \ref{C1}$-i)$ that $u_n \leq m_2\lambda_n^\tau e^t$ for some $m_2$ independent of $n$,  $\tau = (p-1-\delta)^{-1}$ and $t$ as defined before.  As a consequence of this fact and  $-1 < \gamma < 0$, we have  $C \geq  \int_{\Omega}u_n^{\gamma}dx \geq m_2^{\gamma}\lambda_n^{\gamma \tau}\int_{\Omega} e^{t\gamma}dx \rightarrow \infty$, which is an absurd. Therefore, $\int_\Omega  u_n^{\gamma}dx \to \infty$ which implies   $\lambda_n\Big(\int_{\Omega}u_n^{\gamma}dx\Big)^{\theta} \rightarrow 0$, since  $ \theta < 0 $.

 Hence, by using this information together with the hypotheses on $A$, we obtain
$$-\Delta_p u_n \leq C_2 \lambda_n \Big(\int_{\Omega}u_n^{\gamma}dx\Big)^{\theta}(u_n^{\delta} + u_n^{\beta})$$
for some $C_2$ independent of $n$.

Next, by fixing $M = C_2$ and $\alpha = \lambda_n \Big(\int_{\Omega}u_n^{\gamma}dx\Big)^{\theta}$ in (\ref{in1}), we obtain by Corollary \ref{C1}$-i)$ that $u_n \leq m_2\lambda_n^\tau\Big(\int_{\Omega}u_n^{\gamma}dx\Big)^{\theta \tau}  e^t$ for $\tau = (p-1-\delta)^{-1}$, for some $m_2 > 0$ independent of $n$ and for $n$ appropriately large. Therefore, for the choice of $\theta$, we have  $C_3 \geq C_3\Big(\int_{\Omega}u_n^{\gamma}dx \Big)^{1-\tau\theta\gamma} \geq \lambda_n^{\tau\gamma}\rightarrow \infty$ for some $C_3>0$, which leads us to a contradiction again.
\end{itemize}
This ends the proof of Theorem.
\fim
\medskip

To prove Theorem \ref{T4},  let us take advantage of Theorem \ref{T1} to get an unbounded  \textit{continuum} $\Sigma_0$ of positive $W_{\mathrm{loc}}^{1,p}(\Omega)\cap C(\overline{\Omega})$-solutions of
\[
 \left\{
\begin{array}{l}
-\Delta_pu = {\alpha }f(x,u)  ~ \mbox{in } \Omega,\\
    u>0 ~ \mbox{in }\Omega,~~
    u=0  ~ \mbox{on }\partial\Omega,
\end{array}
\right.
\]
  with $Proj_{\mathbb{R}^+}\Sigma_0 = (0, \infty) $. This allows us to  define an appropriated map $H_{\lambda}$ on $\Sigma_0$ such that its zeros are connected with  the solutions of (\ref{autonomo}). More precisely, a pair $(\lambda, u)\in  (0,\infty)\times W_{\mathrm{loc}}^{1,p}(\Omega) \cap C(\overline{\Omega})$ is a solution of (\ref{autonomo}) if and only if  $(\alpha, u) \in \Sigma_0$ with $\alpha = \lambda\Big[A\Big(\displaystyle\int_\Omega  u^{\gamma}dx \Big)\Big]^{-1}$,
which is equivalent to the pair $(\alpha, u) \in \Sigma_0$  being a zero of the map
$$ H_{\lambda}(\alpha, u) = \alpha - \lambda\Big[A\Big(\displaystyle\int_\Omega u^{\gamma}dx \Big)\Big]^{-1} = \Big(\Psi(\alpha, u) - \lambda\Big)\Big[A\Big(\displaystyle\int_\Omega u^{\gamma}dx \Big)\Big]^{-1},~(\alpha,u) \in \Sigma_0 ,$$
where $\Psi(\alpha, u) = \alpha A\Big(\displaystyle\int_\Omega u^{\gamma}dx \Big) $.

Now, we prove the next proposition, which assists us to prove a global existence result for (\ref{autonomo}).
\begin{proposition}
Assume that $- 1 < \gamma < 0$ and  $(A_0)$.  If
  \begin{equation}\label{G}
\displaystyle\limsup_{\stackrel{\alpha \rightarrow 0^+}{(\alpha, u) \in \Sigma_0}} \Psi(\alpha, u) = \infty ~~~\mbox{and} ~~~ \displaystyle\limsup_{\stackrel{\alpha \rightarrow \infty}{(\alpha, u) \in \Sigma_0}} \Psi(\alpha, u) = \infty
\end{equation}
hold,
then there exists a $\lambda^* > 0$ such that $(\ref{autonomo})$ has at least one solution for each  $\lambda \in [\lambda^*, \infty)$ and no solution for $\lambda < \lambda^* $.
\end{proposition}
\noindent\begin{proof}
 As revealed in the proofs of the Claim 1 and Claim 2 of Theorem \ref{T1}, we have $\Sigma_0 \subset \mathcal{F}$, where  $\mathcal{F}$ is defined at ({\ref{cd}}). As a consequence, we conclude that the function $\Psi$ (as above)  is well-defined and continuous on $\Sigma_0$. Let us define
 $$\lambda^* = \inf\{\Psi(\alpha, u) : {(\alpha, u) \in \Sigma_0}\}. $$

 First, we  claim that $\lambda^*>0$. If not, there exists a sequence $\{(\alpha_n, u_n )\} \subset \Sigma_0$  such that $\alpha_nA\Big(\int_{\Omega}u_n^{\gamma}dx \Big) \rightarrow 0,$ which implies by (\ref{G}) that there are positive constants $C_1$ and $C_2$ satisfying $C_1 \leq \alpha_n \leq C_2$.  It follows from  this fact and Corollary \ref{C1}$-ii)$ that  $C_3\phi_1 \leq u_n$ in $\Omega$, for some positive constant $C_3$ independent of $n$, which results in $A\Big(\int_\Omega u_n^{\gamma}dx\Big) \geq C_4 > 0$. As a consequence of this fact and $C_1 \leq \alpha_n \leq C_2$, we have $C_5 \leq \alpha_n A\Big(\int_{\Omega}u_n^{\gamma}dx \Big) $ for some $C_5 > 0$, but this contradicts the fact that  $\alpha_nA\Big(\int_{\Omega}u_n^{\gamma}dx \Big) \rightarrow 0 $.

Next,  let us set $\lambda > \lambda^*$. By definition of $\lambda^*$, we can find a pair $(\alpha^*, u^*) \in \Sigma_0$ satisfying $\lambda^*  < \Psi(\alpha^*, u^*) < \lambda$. On the other hand, it follows from (\ref{G}) that there exists $(\alpha^{**}, u^{**})\in \Sigma_0$ such that $\Psi(\alpha^{**}, u^{**}) > \lambda$. In particular, we have proven that $H_{\lambda}(\alpha^*, u^*) < 0$ and $H_{\lambda}(\alpha^{**}, u^{**}) > 0$. Thus, by Bolzano's Theorem we get the existence of at least one zero of $H_\lambda$ in $\Sigma_0$.

Now, we prove that (\ref{autonomo}) admits at least one solution to $\lambda=  \lambda^*$. For this, it is enough to show that there is a pair $(\alpha, u) \in \Sigma_0$ such that $\Psi(\alpha, u) = \lambda_*$. However, by the definition of $\lambda^*$, we can find a sequence $(\alpha_n, u_n) \subset \Sigma_0$ satisfying $\Psi(\alpha_n, u_n) \rightarrow \lambda^*$. Using the hypothesis (\ref{G}), we again conclude that $C_1 \leq \alpha_n \leq C_2$, up to subsequence, for some positive constants $C_1$ and $C_2$. Thus, following the same argumentation of the proof of the Theorem 1.1, we obtain that $(\alpha_n, u_n ) \rightarrow (\alpha, u) \in \Sigma_0$ in $\mathbb{R}\times C(\overline{\Omega}). $ As $\Psi$ is a continuous application in $\Sigma_0$, we get $\Psi(\alpha, u) = \lambda^*$ as we wanted.

Finally, the non-existence of solutions to $\lambda < \lambda^*$ is a consequence of the definition of $\lambda^*$.  This ends the proof.
\fim
\end{proof}

Through the previous proposition, we are able to prove the Theorem \ref{T4}.

\noindent \textbf{\textbf{\textit{Proof of Theorem $1.3$-Completion:}}} It suffices to verify the hypotheses at (\ref{G}) and apply the above Proposition. To begin with, we prove the first limit at (\ref{G}).
We recall that by Lemma \ref{L3}$-a)$, the inequality $ u \leq \alpha^{\tau}m_2e^t$  holds true whenever $(\alpha, u) \in \Sigma_0$ with $\alpha < \alpha_0$, for some $m_2 >0$ independent of $\alpha$, $\tau = 1/(p-1-\delta)$ and $t= (p-1)/(p-1-\delta)$. By using this inequality and $\gamma < 0$, we get
\begin{equation}\label{bd}
\displaystyle\limsup_{\stackrel{\alpha \rightarrow 0^+}{(\alpha, u) \in \Sigma_0}} \int_\Omega u^{\gamma} = \infty.
\end{equation}

Thus,  as either $(A'_\infty)$ or $(A_\infty)$  with $0 < a_\infty $ holds, it follows from (\ref{bd})  that $$\Psi(\alpha, u) = \alpha A\Big(\int_\Omega u^{\gamma}dx\Big) \geq C_1\alpha\Big(\int_\Omega u^{\gamma}dx\Big)^{-\theta} \geq C\alpha^{1- \tau\theta\gamma } $$ for $\alpha $ small. Since $\theta\gamma > p-1-\delta$, we get $$\displaystyle\limsup_{\stackrel{\alpha \rightarrow 0^+}{(\alpha, u) \in \Sigma_0}} \Psi(\alpha, u) = \infty.$$

Now, let us prove the second limit at (\ref{G}). By  Lemma \ref{L3}$-b)$,  we know that $\alpha^\tau m_1\phi_1 \leq u $  for some $m_1 > 0$ independent of $\alpha$ and for $\tau = 1/(p-1-\beta)$, whenever $(\alpha, u) \in \Sigma_0$ with $\alpha > \alpha_\infty. $  As a result, since $\gamma < 0$, we have
\begin{equation}\label{be}
\displaystyle\limsup_{\stackrel{\alpha \rightarrow \infty}{(\alpha, u) \in \Sigma_0}} \int_\Omega u^{\gamma} = 0.
\end{equation}
   Therefore, by continuity and positivity of $A$ at $t=0$ and (\ref{be}), we obtain  $$\displaystyle\limsup_{\stackrel{\alpha \rightarrow \infty}{(\alpha, u) \in \Sigma_0}} \Psi(\alpha, u) = \infty. $$  This ends the proof.
\fim
\vspace{0.2cm}

Again, let us be benefitted by our tools and follow the strategy of \cite{MR3694626} to approach the problem $(P)$ for the degenerate case, that is, when $ A (x, 0) = 0 .$ This procedure allows us to complement the results in \cite{MR3694626} both to $p$-Laplacian operator, with $1<p<\infty$, and strongly-singular non-linearities.

\begin{theorem}[Degenerate case: A(x, 0) = 0]
 Assume that $\gamma > 0$ and $f$ satisfies $(f_1)$, $(f_2)$ with $\delta \leq \beta$. If  $A \in C(\overline{\Omega}\times[0, \infty), [0, \infty))$  with  $A(x, 0) = 0$ in $\Omega$,  $\theta\gamma = p-1- \beta$ and:
\begin{itemize}
\item[$a)$] $(A'_\infty)$ holds, then $(P)$ has at least one solution for each $\lambda > 0$.
\item[$b)$] $(A_\infty)$ holds with $0 < a_\infty $ in $\overline{\Omega}$, then $(P)$ has at least one solution for $\lambda $ small and no solution for $\lambda $ large.
\end{itemize}
\end{theorem}
\begin{proof}
For each $n \in \mathbb{N}$, consider
$$
 (P_n)\left\{
\begin{array}{l}
-A_{n}\Big(x, \displaystyle\int_{\Omega}u^{\gamma}dx\Big) \Delta_pu = {\lambda }f(x,u)  ~ \mbox{in } \Omega,\\
    u>0 ~ \mbox{in }\Omega,~~
    u=0  ~ \mbox{on }\partial\Omega,
\end{array}
\right.
$$
where $A_n(x,t) = A(x,t) + 1/n. $ Since $\displaystyle\lim_{t\rightarrow \infty}A_{n}(x,t)t^{\theta} = \infty$, with $\theta \gamma = p-1- \beta$, it follows from  the item $a)$ of Theorem \ref{T2}  that $(P_n)$ has at least one solution for each $\lambda > 0$. Thus, given a $\lambda > 0$, denote by $u_n$ one such solution of $(P_n)$. From this, let us prove the items $a)$ and $b)$ above.
\begin{itemize}
\item[a)] The proof of this item is a consequence of the following claims:

 \begin{equation}\label{ns} i) \int_\Omega u_n^{\gamma} dx \not\rightarrow 0 ~~~ ~~~~\mbox{and} ~~ ~~~~ii) \int_\Omega u_n^{\gamma} dx \not\rightarrow \infty. \end{equation}

Let us prove the first claim in (\ref{ns}).
Suppose by contradiction, that $\int_\Omega u_n^{\gamma} dx \rightarrow 0$. Since $A(x,0) = 0$ and $A$ is a continuous function, for  given  $C > 0$  sufficiently large there exists $n_0 \in \mathbb{N}$  such that  $A_{n}\Big(x, \int_\Omega u_n^{\gamma}dx \Big) < 1/C $ for all $n > n_0 $. Thus, we get  $-\Delta_p u_n \geq \lambda C f(x, u_n)$, which implies by
Corollary \ref{C1}$-ii)$ that  $u_n \geq (\lambda C)^\tau m_1 \phi_1$  for  $n$ large, where $\tau = (p-1-\beta)^{-1}$. Hence, from this inequality we get $0 < (\lambda C)^{\tau\gamma}m_1^{\gamma}\int_\Omega \phi_1^{\gamma}dx  \leq \int_\Omega u_n^{\gamma}dx \rightarrow 0, $ which is an absurd.

Now we will prove the second claim in (\ref{ns}).
Again, suppose by contradiction that $\int_\Omega u_{n}^{\gamma}dx \rightarrow \infty$. From ($A'_\infty$), for each $C > 0$ enough large, we have $A\Big(x, \int_\Omega u_n^{\gamma}dx \Big)\Big(\int_\Omega u_n^{\gamma}dx \Big)^{\theta} > C$ for all $n $ big enough. In this case, we obtain  $-\Delta_p u_{n} \leq \frac{\lambda}{C}\Big(\int_\Omega u_n^{\gamma}dx \Big)^{\theta}f(x, u_n), $ which by the Corollary \ref{C1}$-i)$ and simple calculations implies
\begin{equation}\label{I}
\Big(\int_{\Omega}u_{n}^{\gamma}dx \Big)^{1-\tau\theta\gamma } \leq \Big(\frac{\lambda}{C}\Big)^\tau m_2^{\gamma},
\end{equation}
where $\tau = (p-1- \beta)^{-1}. $
As $\theta\gamma = p-1- \beta$ and  $C>0$ was taken large enough, the inequality (\ref{I}) results into $1 \leq \Big(\frac{\lambda}{C}\Big)^\tau m_2^\tau < 1$.  This is an absurd and from this the Claim in $ii)$ is proved.

Observe that from claims in $i)-ii)$, we get $0 < C_1 \leq \int_\Omega u_{n}^{\gamma}dx \leq C_2$, for some positive constants $C_1$ and $C_2$.  Thus, proceeding as in the proof of the Claim 2 in  Theorem \ref{T1}, we can  show that $u_n$ converge in $W_{\mathrm{loc}}^{1,p}(\Omega)$ for some $u \in W_{\mathrm{loc}}^{1,p}(\Omega) \cap C(\overline{\Omega}),$ which is a solution of $(P)$. It concludes the proof the item$-a)$.

\item[b)]
As in the item-$a)$, the proof here follows from the following asserts:
\begin{equation}\label{nS}
i) \int_\Omega u_n^{\gamma} dx \not\rightarrow 0 ~~~~~~~~~\mbox{and} ~~~~~~~ ~~~ii) \int_\Omega u_n^{\gamma} dx \not\rightarrow \infty, \mbox{for each} ~\lambda > 0 ~\mbox{small.}
\end{equation}

The proof of the first Claim in (\ref{nS}) is the same of the previous item$-i)$.

Let us prove $ii)$. As $a_\infty > 0$ in $\overline{\Omega}$, then defining $C = (\inf_\Omega a_\infty)/2$, there exists $ t_0 > 0$ such that $A(x,t)t^{\theta} \geq C > 0 $ for all $ t > t_0$. Thus, if we suppose that $\int_\Omega u_n^{\gamma}dx \rightarrow \infty$, we obtain  $-\Delta_p u_{n} \leq \frac{\lambda}{C}\Big(\int_\Omega u_n^{\gamma}dx\Big)^{\theta}f(x, u_n),$ which again by Corollary \ref{C1}$-i)$ implies in  $u_n \leq \Big(\frac{\lambda}{C}\Big)^\tau\Big(\int_\Omega u_n^{\gamma }dx\Big)^{\theta \tau}m_2e^t$ for some $m_2 > 0,$ $\tau = (p-1- \beta)^{-1}$, $t= (p-1)/(p-1-\delta)$ and $n$ appropriately large. As a consequence of this, we obtain  $ \Big(\int_\Omega u_n^{\gamma }dx\Big)^{1 - \theta\gamma \tau} \leq \Big(\frac{\lambda}{C}\Big)^{\gamma \tau}m_2^{\gamma}\int_\Omega e^{t\gamma}dx.$ Since $\theta\gamma = p-1-\beta$, we get by the last inequality that $1 \leq \Big(\frac{\lambda}{C}\Big)^{\gamma \tau}m_2^{\gamma}\int_\Omega e^{t\gamma}dx, $ however this is a contradiction for $\lambda < {C}{\Big(m_2^{\gamma}\int_\Omega e^{t\gamma}dx\Big)^{-1/\gamma \tau}} = \lambda^*. $ Therefore, $\int_\Omega u_n^{\gamma} dx \not\rightarrow \infty$ for $0 < \lambda < \lambda^*$.

From $i)-ii)$, by the same argument used in item$-a)$ we conclude that $(P)$ admits at least one positive solutions for $0 < \lambda < \lambda^*$. To justify that $(P)$ does not have solution for  $\lambda$ large, just follow the same argument of item $b$) of Theorem \ref{T2}, using $\theta\gamma = p-1-\beta. $
\end{itemize} This proves the Theorem.
\fim
\end{proof}

\section{A strongly-singular non-autonomous Kirchhoff problem}

In this section, we prove Theorem \ref{T5} which deals with a non-autonomous Kirchhoff problem, defined in $(Q)$,  with strongly-singular nonlinearity.

The proof of  Theorem \ref{T5} follows the same steps of  Theorem \ref{T1} with small adaptations. Recall that in the proof of  Lemma \ref{L2} we used that $\
\|u_\epsilon\|_\gamma \leq C$ for  some $C$ independent of $\epsilon$, where  $(\lambda_\epsilon, u_\epsilon)$ is a solution of perturbed problem $(P_\epsilon)$ and belongs to the boundary of an open bounded set containing $(0,0)$. Here, due to the presence of  $\|\nabla u\|_p$ in the Kirchhoff term, we need a similar estimate on $\|\nabla u_\epsilon\|_p$, which is crucial in our argument. To avoid repetition, we present a sketch  of each step while giving attention to the notable points.
Corresponding to $(Q)$, we introduce the following perturbed problem
$$
 (Q_{\epsilon})~ ~ ~ \left\{
\begin{array}{l}
-M\Big(x, \|\nabla u\|_p^p\Big)\Delta_pu = {\lambda }f(x,u + \epsilon)  ~ \mbox{in } \Omega,\\
    u>0 ~ \mbox{in }\Omega,~~
    u=0  ~ \mbox{on }\partial\Omega.
\end{array}
\right.
$$
About $(Q_\epsilon)$, we have the following result.
 \begin{lemma}
  Suppose that  $\gamma > 0$ and $M$ satisfies $(M_0)$. Then, for each $\epsilon > 0$ there exists an unbounded $\epsilon-$\textit{continuum} $\Sigma_{\epsilon} \subset \mathbb{R}^+ \times C(\overline{\Omega})$ of positive solutions of $(Q_{\epsilon})$ emanating from $(0,0)$.
 \end{lemma}

\noindent\begin{proof}
Consider for each $\lambda, R > 0$ and $v \in C(\overline{\Omega})$, the auxiliary problem
\begin{equation}\label{H2}
\left\{
\begin{array}{l}
-M(x, R)\Delta_pu = {\lambda }f(x,|v| + \epsilon)  ~ \mbox{in } \Omega,\\
    u>0 ~ \mbox{in }\Omega,~~
    u=0  ~ \mbox{on }\partial\Omega.
\end{array}
\right.
\end{equation}
As $M(x,t)=a(x)+b(x)t^{\gamma}$ with $a(x) \geq \underline{a} > 0$ and $f$ is continuous, then (\ref{H2}) admits a unique solution $u_R \in C^{1,\alpha}(\overline{\Omega}) \cap W_0^{1,p}(\Omega)$, for some $\alpha \in (0,1)$. Thus
$$  \displaystyle\int_{\Omega}|\nabla u_R|^pdx = \displaystyle\int_\Omega \frac{\lambda f(x, |v| + \epsilon)u_R}{M(x,R)}dx. $$

Define $h: \mathbb{R}^+ \rightarrow \mathbb{R}^+$  by $h(R) = \displaystyle\int_\Omega \frac{\lambda f(x, |v| + \epsilon)u_R}{M(x,R)}dx.$ Note that $h$ is continuous and $h(0) > 0$. Moreover, observe that $h$ is non-increasing. Indeed, if $R_1 < R_2$ then
$$-\Delta_p u_{R_2} =  \frac{\lambda f(x, |v| + \epsilon)}{M(x,R_2)} \leq \frac{\lambda f(x, |v| + \epsilon)}{M(x,R_1)} = -\Delta_pu_{R_1}. $$
Also, as $u_{R_1}|_{\partial \Omega} = u_{R_2}|_{\partial \Omega}$,
from classical comparison principle, we have $u_{R_2} \leq u_{R_1} $ and as a consequence we conclude that $h(R_2) \leq h(R_1)$.  Thus,  there exists a unique solution (say $\tilde{R}$) of $h({R}) = {R}, $ that is,
$$\tilde{R} = \displaystyle\int_{\Omega}\frac{\lambda f(x, |v|+\epsilon)u_{\tilde{R}}}{M(x, \tilde{R})}dx  = \displaystyle\int_{\Omega}|\nabla u_{\tilde{R}}|^p dx. $$ Hence,  $u_{\tilde{R}}$ is a solution of
\begin{equation}\label{H3}
\left\{
\begin{array}{l}
-M\Big(x, \|\nabla u\|_p^p)\Delta_pu = {\lambda }f(x,|v| + \epsilon)  ~ \mbox{in } \Omega,\\
    u>0 ~ \mbox{in }\Omega,~~
    u=0  ~ \mbox{on }\partial\Omega.
\end{array}
\right.
\end{equation}

We claim that (\ref{H3}) has unique solution. In fact, suppose that $u \neq w \in W_0^{1,p}(\Omega)$ are two solutions of (\ref{H3}). If $\int_{\Omega}|\nabla u|^pdx = \int_{\Omega}|\nabla w|^pdx$, then $u = w$ in $\Omega$.  On the other hand, if
$R_1 = \int_{\Omega}|\nabla u|^pdx <  \int_{\Omega}|\nabla w|^pdx = R_2$, we have $u_{R_2} \leq u_{R_1} $ and as a consequence
 $$R_2=\displaystyle\int_{\Omega}|\nabla w|^pdx = \displaystyle\int_{\Omega}\frac{f(x,|v|+\epsilon)u_{R_2}}{M(x,R_2)} dx \leq \displaystyle\int_{\Omega}\frac{f(x,|v|+\epsilon)u_{R_1}}{M(x,R_1)} dx =  \displaystyle\int_{\Omega} |\nabla u|^pdx=R_1.$$
 Therefore, in any case we get a contradiction, which proves  that (\ref{H3}) has only one solution. Now, we consider the operator $T: \mathbb{R}^+\times C(\overline{\Omega}) \rightarrow C(\overline{\Omega})$ which associates each pair $(\lambda, v) \in \mathbb{R}^+\times C(\overline{\Omega})$ to the only solution of (\ref{H3}). Since $M(x,t) \geq \underline{a} > 0 \in \Omega$, the rest of the proof follows from Lemma \ref{L1}, in a similar way.
 \fim
\end{proof}
\vspace{0.2cm}

In order to study the limit behavior of the components $\Sigma_\epsilon,$ we prove the following Lemma.

\begin{lemma}\label{L7} Suppose $(f_2)$, $(M_0)$ and $(\Gamma_0)$ holds. Let
 $U \subset \mathbb{R} \times C(\overline{\Omega})$ be a bounded open set containing $(0,0)$ and $(\lambda_{\epsilon}, u_{\epsilon})$ be a solution of $(Q_{\epsilon})$ such that $ (\lambda_{\epsilon}, u_{\epsilon})\in \Sigma_\epsilon \cap \Big((0, \infty) \times W_0^{1,p}(\Omega))\Big) \cap \partial U$. Then, for some positive constant $C(U)$, independent of $\epsilon$, we have $\|\nabla u_\epsilon\|_p \leq C(U)$.
\end{lemma}

\noindent\begin{proof}
Consider $(\lambda_\epsilon, u_\epsilon) \in \Sigma_\epsilon \cap \partial U$, then $\lambda_\epsilon\leq K, \|u_\epsilon\|_{\infty} \leq K$ for some positive constant $K$ depending only on $U$. Taking $u_{\epsilon}$ as a test function in $(Q_\epsilon)$ and using  $(f_2)$ we get
\begin{equation}\label{H9}
 \|\nabla u_\epsilon\|_p^p \leq C_1\lambda_{\epsilon}\Big(\displaystyle\int_\Omega (u_{\epsilon}+ \epsilon)^{\delta + 1}dx + 1\Big).
\end{equation}
If $\delta \geq -1$, then by (\ref{H9}) the required boundedness follows trivially from the fact that $\lambda_{\epsilon}\leq K,  \|u_{\epsilon}\|_{\infty} \leq K. $ Now, suppose that $\delta \in \Big(-\frac{2p-1}{p-1}, -1\Big). $
As $\|u_{\epsilon}\|_{\infty} \leq K$, by the continuity of $f$ we can find a $C_2 > 0$ independent of $\epsilon$ such that $f(u_{\epsilon} + \epsilon) \geq C_2(u_\epsilon + \epsilon)^{\delta}$. Thus, $u_\epsilon + \epsilon$ is a supersolution of
\begin{equation}\label{H8}
-\Delta_p u = \frac{\lambda_{\epsilon}C_2u^{\delta}}{\displaystyle\max_{\overline{\Omega}} a + \displaystyle\max_{\overline{\Omega}}b \|\nabla u_{\epsilon}\|_p^{\gamma p}}.
\end{equation}
On the other hand, take $\underline{u} = s\phi_1^{\frac{p}{p-1-\delta}}, $ where $s>0$ will be fixed later, then a simple calculation shows that

\begin{eqnarray*}
 -\Delta_p \underline{u}& = & \Big(\frac{sp}{p-1-\delta}\Big)^{p-1}\phi_1^{\frac{\delta p}{p-1-\delta}}\Big[\frac{(-\delta -1)(p-1)}{p-1-\delta}|\nabla \phi_1|^p + \lambda_1\phi_1^p\Big] \\
 &\leq & C_3\Big(\frac{sp}{p-1-\delta}\Big)^{p-1}\phi_1^{\frac{\delta p}{p-1-\delta}} = C_3s^{p-1-\delta}\Big(\frac{p}{p-1-\delta}\Big)^{p-1}\underline{u}^{\delta},
\end{eqnarray*}
where $C_3 = \displaystyle\max_{\overline{\Omega}}\Big[\frac{(-\delta -1)(p-1)}{p-1-\delta}|\nabla \phi_1|^p + \lambda_1\phi_1^p\Big]. $  Therefore, if we  choose $$ s = C_4\Big(\frac{\lambda_{\epsilon}}{\displaystyle\max_{\overline{\Omega}} a + \displaystyle\max_{\overline{\Omega}}b \|\nabla u_{\epsilon}\|_p^{\gamma p}}\Big)^{\frac{1}{p-1-\delta}},$$ where
$C_4 = \Big[\frac{C_2(p-1-\delta)^{p-1}}{C_3p^{p-1}}\Big]^{\frac{1}{p-1-\delta}}, $ then $\underline{u}$ is a subsolution of (\ref{H8}) and  by the Theorem \ref{unicidade1} we get
\begin{equation}\label{el}
u_{\epsilon} + \epsilon \geq C_4\Big(\frac{\lambda_{\epsilon}}{\displaystyle\max_{\overline{\Omega}} a + \displaystyle\max_{\overline{\Omega}}b \|\nabla u_{\epsilon}\|_p^{\gamma p}}\Big)^{\frac{1}{p-1-\delta}}\phi_1^{\frac{p}{p-1-\delta}}.
\end{equation}
Now, coming back in (\ref{H9}) and using (\ref{el}) together with $\delta \in \Big(-\frac{2p-1}{p-1}, -1\Big)$, we obtain
$$ \|\nabla u_{\epsilon}\|_p^p \leq C_5\Big(1 + \|\nabla u_{\epsilon}\|_p^{-\frac{\gamma p(\delta + 1)}{p-1-\delta}} \Big).$$
Since $\gamma < \frac{p-1-\delta}{-1-\delta},$  it follows from the last inequality  that $\|\nabla u_{\epsilon}\|_p \leq C(U)$, where $C(U)$ is independent of $\epsilon$.
\fim
\end{proof}

In the light of above result, we prove the following Lemma, similar to Lemma \ref{L2}. We highlight only the principal points in the proof.

\begin{lemma}\label{L5} Admit that $f$, $M$ and $\gamma$ satisfy  $(f_2)$, $(M_0)$ and $(\Gamma_0)$, respectively. Let $U \subset \mathbb{R} \times C(\overline{\Omega})$ be a bounded open set containing $(0,0)$ and a pair $ (\lambda_{\epsilon}, u_{\epsilon}) \in \Sigma_\epsilon \cap\Big((0, \infty) \times (C(\overline{\Omega}) \cap W_0^{1,p}(\Omega))\Big) \cap \partial U$ be a solution of $(Q_{\epsilon})$ satisfying $\lambda_\epsilon\leq K, \Vert u_\epsilon\Vert_{\infty} \leq K$. Then, there are positive constants $\mathcal{K}_1 = \mathcal{K}_1(K, U)$, $\mathcal{K}_2 = \mathcal{K}_2(k,K)$ and $\epsilon_0 > 0$ such that
\begin{equation}\label{H6}
\lambda_{\epsilon}^{\frac{1}{p-1}}\mathcal{K}_1(K, U) \phi_1 \leq u_{\epsilon} \leq k + \lambda_{\epsilon}^{\frac{1}{p-1}}\mathcal{K}_2(k, K)^{\frac{1}{p-1}}e ~~ \mbox{in} ~\Omega
\end{equation}
for each $k\in (0, K]$ fixed and for all $0 < \epsilon < \epsilon_0$.
\end{lemma}
\begin{proof}
Define $\mathcal{K}_2(k,K) = \max\Big\{\frac{f(x,t)}{\underline{a}} ~x \in \overline{\Omega}: ~k \leq t \leq K+1\Big\}$, where $k \in (0, K]$. For this constant, a second inequality in (\ref{H6}) holds.

To obtain the first inequality, we must proceed as in the proof of the first inequality in Lemma \ref{L2}. To get the constant $\mathcal{K}_1(K, U)$, in (\ref{9}) we choose  $A'_U := \max\{M(x,t): x \in \overline{\Omega} ~\mbox{and} ~ 0 \leq t \leq C(U)^p\}$ instead of $ A_K$, where $C(U)$ is given in the Lemma \ref{L7}.
\fim
\end{proof}

Now we are ready to prove the Theorem \ref{T5}.
\vspace{0.2cm}

\textbf{\textit{Proof of Theorem \ref{T5}:}}
Suppose that $\epsilon_n \rightarrow 0^+$ and denote by $\Sigma_n \subset \mathbb{R}^+ \times C(\overline{\Omega})$ the  component associated with the problem $(Q_{\epsilon_n})$. Let  $U \subset \mathbb{R} \times C(\overline{\Omega})$ be an open neighborhood of $(0,0)$. As $\Sigma_n$ is unbounded, there exists $(\lambda_n, u_n) \in \Sigma_n \cap \partial U$ and $K> 0$ such that $\lambda_n\leq K, \|u_n\|_{\infty} \leq K$. Moreover, from Lemma \ref{L7} we can assume, without loss of generality, that $\|\nabla u_n\|_p^p \leq K$ and from Lemma \ref{L5} that $\lambda_n \rightarrow \lambda > 0^+$, up to a subsequence. As a consequence, for $\delta'> 0$ small there exists $n_0 \in \mathbb{N}$ such that $0 < \lambda - \delta' < \lambda_n < \lambda + \delta' $ for all $n \geq n_0$, which implies again by the Lemma \ref{L5} that
\begin{equation}\label{H10}
(\lambda - \delta')^{1/(p-1)}\mathcal{K}_1(K, U)\phi_1 \leq u_n \leq k + (\lambda + \delta')^{1/(p-1)}\mathcal{K}_2(k,K)^{1/(p-1)}e ~~\mbox{in} ~\Omega,
\end{equation}
for each $k \in (0, K]$.

From Lemma \ref{L7}, $\{u_n\}$ being bounded in  $W_0^{1,p}(\Omega)$,
there exists $u=u_\lambda \in W_0^{1,p}(\Omega)$ such that  $u_n \rightharpoonup u$ in $W_0^{1,p}(\Omega)$ weakly. Proceeding as in the proof of Theorem \ref{T1}, we conclude by (\ref{H10}) that $u$ satisfies
\begin{equation}\label{H11}
\displaystyle\int_{\Omega}|\nabla u|^{p-2}\nabla u\nabla \varphi dx = \lambda\displaystyle\int_{\Omega}\frac{f(x,u)}{M(x, \|\nabla u\|_p^p)}\varphi dx, ~~\mbox{for all} ~\varphi \in C_c^{\infty}(\Omega).
\end{equation}
Let us prove that (\ref{H11}) holds also for $\varphi \in W_0^{1,p}(\Omega)$. For this, take $\varphi \in W_0^{1,p}(\Omega)$. Then, by a density results, there exists a sequence $\{\varphi_n\} \in C_c^{\infty}(\Omega)$ such that $\varphi_n \rightarrow \varphi$ in $W_0^{1,p}(\Omega)$.  Now, for each $\epsilon > 0$ the function $\phi = \sqrt{\epsilon^2 + |\varphi_n -
 \varphi_k|^2} + \epsilon \in C_c^1(\Omega)$ and hence taking $\phi$ as a test function in (\ref{H11}), we obtain
 \begin{eqnarray*}
 \lambda\displaystyle\int_{\Omega}\frac{f(x,u)}{M(x, \|\nabla u\|_p^p)}\Big(\sqrt{\epsilon^2 + |\varphi_n -\varphi_k|^2} - \epsilon\Big)dx & = & \displaystyle\int_{\Omega}|\nabla u|^{p-2}\nabla u \frac{|\varphi_n - \varphi_k|\nabla(\varphi_n - \varphi_k)}{\sqrt{\epsilon^2 + |\varphi_n -\varphi_k|^2}}dx \\
 &\leq & \displaystyle\int_{\Omega}|\nabla u|^{p-1}|\nabla(\varphi_n - \varphi_k)|dx \\
 &\leq & C \|\nabla u\|_p^{p-1}\|\nabla(\varphi_n - \varphi_k)\|_p.
\end{eqnarray*}
Applying the Fatou's Lemma, we obtain by the previous inequality
\begin{eqnarray*}  \lambda\displaystyle\int_{\Omega}\frac{f(x,u)}{M(x, \|\nabla u\|_p^p)} |\varphi_n -\varphi_k|dx &\leq & \displaystyle\liminf_{\epsilon \rightarrow 0^+} \lambda\displaystyle\int_{\Omega}\frac{f(x,u)}{M(x, \|\nabla u\|_p^p)}\Big(\sqrt{\epsilon^2 + |\varphi_n -\varphi_k|^2} - \epsilon\Big)dx \\ &\leq&  C \|\nabla u\|_p^{p-1}\|\nabla(\varphi_n - \varphi_k)\|_p.
\end{eqnarray*}
Letting $n, k \rightarrow \infty$ in the previous inequality we obtain
$$ \lambda\displaystyle\int_{\Omega}\frac{f(x,u)}{M(x, \|\nabla u\|_p^p)} |\varphi_n -\varphi_k|dx \rightarrow 0.  $$ Thus, we have
\begin{equation}\label{H15}
\displaystyle\int_{\Omega}\frac{f(x,u)}{M(x, \|\nabla u\|_p^p)} \varphi_n dx {\longrightarrow} \lambda\displaystyle\int_{\Omega}\frac{f(x,u)}{M(x, \|\nabla u\|_p^p)} \varphi dx ~\mbox{as} ~{n \rightarrow \infty}.
\end{equation}
By classical density arguments, we also have
\begin{equation}\label{H16}
\displaystyle\int_{\Omega}|\nabla u|^{p-2}\nabla u\nabla \varphi_n dx {\longrightarrow} \displaystyle\int_{\Omega}|\nabla u|^{p-2}\nabla u\nabla \varphi dx ~\mbox{as} ~{n \rightarrow \infty}.
\end{equation}
Therefore, joining (\ref{H15}) and (\ref{H16}) we obtain that $u \in W_0^{1,p}(\Omega)\cap C(\overline{\Omega})$ is solution of $(Q)$ and satisfies (\ref{H10}).

Now, if we consider $\mathcal{F}$ as in the proof of Theorem \ref{T1}, then in a similar way we can show that closed and bounded ( in $\mathbb{R}\times C(\overline{\Omega}))$ subsets of $\mathcal{F}$ are compacts and this ends the proof of existence of the unbounded \textit{continuum} $\Sigma$.

The proof of $Proj_{\mathbb{R}^+} \Sigma = (0, \infty) $  if $(f_\infty)$ holds, is the same as done in the proof of Theorem \ref{T1}.

Now, suppose that there exists a constant $C>0$, independent of $\lambda$ and $u$, such that $\|u\|_{\infty} \leq C$  whenever  $(\lambda, u) \in \Sigma$. Then, let us take $(\lambda, u) \in \Sigma$ with $\lambda > 1$, so $u $ satisfies
$$ -\Delta_p u \geq \frac{\lambda C_1}{\displaystyle\max_{\overline{\Omega}} a(x) + \displaystyle\max_{\overline{\Omega}} b(x) \|\nabla u\|_p^{p\gamma}} u^{\delta}. $$

Besides this, for $\epsilon > 0$ small $\underline{u} = \Big(\frac{{{\epsilon \lambda}}}{\displaystyle\max_{\overline{\Omega}} a(x) + \displaystyle\max_{\overline{\Omega}} b(x) \|\nabla u\|_p^{p\gamma}}\Big)^{1/(p-1-\delta)} \phi_1^{\frac{p}{p-1-\delta}} $ satisfies
$$-\Delta_p \underline{u} \leq  \frac{\lambda C_1}{\displaystyle\max_{\overline{\Omega}} a(x) + \displaystyle\max_{\overline{\Omega}} b(x) \|\nabla u\|_p^{p\gamma}} \underline{u}^{\delta}, $$ and so we get by Theorem \ref{unicidade1} that
$u\geq \underline{u}$. Taking $u$ as a test function in $(Q)$ and using $\lambda > 1$, $u\geq \underline{u}$ and $\|u\|_{\infty} \leq C$, we obtain that
\begin{equation}\label{H20}
\left\{
\begin{array}{l}
\displaystyle\int_{\Omega}|\nabla u|^pdx \leq C_1\lambda ~~~ \mbox{if} ~\delta \geq -1 \\
\displaystyle\int_{\Omega}|\nabla u|^pdx \leq C\lambda^{\frac{p}{p-1-\delta}}(\|\nabla u\|_p^{\frac{p(-\delta -1)\gamma}{p-1-\delta}} + 1) ~~~ \mbox{if} ~ -\frac{2p-1}{p-1} < \delta < -1.
\end{array}
\right.
\end{equation}
Without loss of generality, let us assume that $\|\nabla u\|_p > 1$, otherwise we would get
$$C\geq u\geq \underline{u} \geq  \Big(\frac{{{\epsilon \lambda}}}{\displaystyle\max_{\overline{\Omega}} a(x) + \displaystyle\max_{\overline{\Omega}} b(x) }\Big)^{1/(p-1-\delta)} \phi_1^{\frac{p}{p-1-\delta}}~\mbox{for all }\lambda>0. $$
Then, coming back to (\ref{H20}) and using $\|\nabla u\|_p > 1$, we obtain  for $ -\frac{2p-1}{p-1} < \delta < -1$ that $\|\nabla u\|_p \leq C\lambda^{\frac{1}{p +(\delta+1)(\gamma -1)}}$. Thus, as $u \geq \underline{u}$ we have
\begin{equation}\label{H17}
u \geq   C\Big(\frac{{{ \lambda}}}{1 +  \lambda^{\frac{p\gamma}{p+(\delta+1)(\gamma -1)}}}\Big)^{1/(p-1-\delta)} \phi_1^{\frac{p}{p-1-\delta}}.
\end{equation}
Also, when $\delta \geq -1$ by (\ref{H20}) we get
\begin{equation}\label{H19}
u \geq  C\Big(\frac{{{ \lambda}}}{1 +  \lambda^{\gamma}}\Big)^{1/(p-1-\delta)} \phi_1^{\frac{p}{p-1-\delta}}.
\end{equation}
Then, from (\ref{H17}) and (\ref{H19}) with $ \gamma < 1$, it follows that  $\|u\|_{\infty} {\rightarrow} \infty $ as $\lambda \rightarrow \infty$, contradicting the fact that $\|u\|_{\infty} \leq C$.
\fim

\end{document}